\documentclass[a4paper,10pt]{article}
\usepackage[T2A]{fontenc}
\usepackage[english]{babel}

\usepackage{graphicx}
\usepackage{amsthm, amsmath, amssymb, amsbsy, amsfonts, dsfont}
\usepackage{srcltx, color}

\def\tht{\theta}
\def\Om{\Omega}
\def\om{\omega}
\def\e{\varepsilon}
\def\g{\gamma}

\def\l{\lambda}
\def\p{\partial}

\def\vk{\kappa}

\def\a{\alpha}
\def\b{\beta}

\def\d{\delta}
\def\L{\Lambda}
\def\z{\zeta}
\def\vp{\varphi}

\def\vt{\vartheta}

\def\la{\langle}
\def\ra{\rangle}

\def\iu{\mathrm{i}}

\def\cM{\mathcal{M}}
\def\cL{\mathcal{L}}

\def\cP{\mathcal{P}}

\def\Ups{\Upsilon}

\def\cH{\mathcal{H}}

\def\cL{\mathcal{L}}

\DeclareMathOperator*{\esssupp}{ess\,supp} 

\DeclareMathOperator{\supp}{supp}
\DeclareMathOperator{\RE}{Re}
\DeclareMathOperator{\IM}{Im}
\DeclareMathOperator{\dist}{dist}

\DeclareMathOperator{\mes}{mes}
\DeclareMathOperator{\diam}{diam}

\numberwithin{equation}{section}

\newtheorem{theorem}{Theorem}[section]
\newtheorem{lemma}{Lemma}[section]

\begin{document}

\allowdisplaybreaks
\title{Nonlocal convolution type functionals and related Orlicz spaces}

\date{\empty}

\author{D.I. Borisov$^{1,2}$, A.L. Piatnitski$^{3,4}$}

\maketitle
{\small
 \begin{quote}
1) Institute of Mathematics, Ufa Federal Research Center, Russian Academy of Sciences, Chernyshevsky str. 112, Ufa, Russia, 450008
\\
2) Peoples Friendship University of Russia (RUDN University), 6 Miklukho--Maklaya Street, Moscow, 117198, Russian Federation
\\
3) The Arctic University of Norway, campus Narvik, PO Box 385, N-8505 Narvik, Norway
\\
4) Higher School of Modern Mathematics MIPT, 1 Klementovski per., Moscow, Russia
\\[2mm]
Emails: borisovdi@yandex.ru, apiatnitski@gmail.com
\end{quote}}

{\small \begin{quote}
 \noindent{\bf Abstract.}   In the paper we introduce Orlicz type functional spaces defined in terms of nonlocal convolution
 type integral functionals and study the main properties of these spaces.  We show in particular that, under natural convexity
 and growth conditions on the integrand, the corresponding spaces are Banach and separable. We also characterize the dual spaces and provide a number of examples.

 \medskip


 \medskip

 \end{quote}}

\section{Introduction}\label{s_intro}

The main goal of this paper is to introduce and study the 
functional spaces defined in terms of nonlocal
``convolution type'' integral functionals 
\begin{equation}\label{intr_functional}
F(u)=\int\limits_{\Omega\times\Omega}a(x-y) \vp(|u(x)-u(y)|,x,y)\,dxdy,
\end{equation}
which include the important particular case
\begin{equation}\label{intr_part_case}
F_{p(\cdot)}(u)=\int\limits_{\Omega\times\Omega}\frac{a(x-y) b(x,y)}{p(x,y)} |u(x)-u(y)|^{p(x,y)}\,dxdy.
\end{equation}
Here $\Omega$ is 
a regular domain in $\mathds{R}^d$ that can be bounded or unbounded including the case, when $\Omega$ coincides with $\mathds{R}^d,$
and $a(z)$ is a non-negative integrable function.  
The function $\phi(z,x,y)$ is assumed to be non-negative and
strictly convex in the first variable and to  satisfy appropriate growth conditions that can vary from point to point;
the detailed list of conditions is formulated in the next section.
For the  particular case \eqref{intr_part_case}, these conditions are satisfied if $b$ is a bounded  positive measurable function and $p(x,y)$ is a measurable function obeying the inequality $1<p_-\leqslant p(\cdot)\leqslant p_+.$


Variational nonlocal convolution type functionals with variable growth conditions is a very interesting mathematical topic.
It is closely related to the theory of  convolution type functionals with variable growth conditions
and the corresponding stationary and evolution equations.
Studying both qualitative and asymptotic properties of these operators and functionals is an important and challenging mathematical task.


The interest to the nonlocal convolution type functionals and operators is also motivated  by numerous applications in material
sciences and biology.   They are widely used for describing various processes in population dynamics, mechanics of porous media, chemistry of polymers, and other fields, where the non-locality of interactions plays an essential role.
For instance, in the so-called contact
model in population dynamics the evolution of the density of population is characterized by the equation
$$
\partial_t v(x,t)=\int\limits_{\mathds{R}^d} a(x-y)b(x,y) \big(v(y,t)-v(x,t)\big)dy.
$$
The quadratic form corresponding to the operator on the right-hand side here is given by
\eqref{intr_part_case} with $\Omega=\mathds{R}^d$ and $p(\cdot)=2$. The description of the contact model and
its basic properties can be found in \cite{KKP}.
 In order to provide a more accurate description one should consider
the models that also account for the nonlinear effects. In this case the corresponding variational functional is
not quadratic any more and might be of the form \eqref{intr_part_case} or even \eqref{intr_functional}.

In the modern theory of porous media, the modeling of many phenomena is performed by taking into account the non-locality and non-linearity of interactions in the medium.
Moreover,  the rheological properties of the media can vary from point to point.
Mathematical description of these phenomena relies on nonlinear convolution type equations; the example of such equation is
$$
\partial_t v(x,t)=\int\limits_{\mathds{R}^d} a(x-y)b(x,y) \big(v(y,t)-v(x,t)\big)\big|v(y)-v(x)\big|^{p(x,y)-2}dy,
$$
and the variational functional corresponding to the operator on the right-hand side coincides with that in
\eqref{intr_part_case}.

The local functionals and differential operators with variable growth conditions have been intensively studied
during last thirty years, see, for instance,  \cite{DHHR11}, \cite{HHLN10}, \cite{Ru00} and references therein.
We recall that in order to investigate elliptic PDEs with variable exponents of the form
\begin{equation}\label{perem_diff}
-\operatorname{div}
A(x)|\nabla v(x)|^{p(x)-2}\nabla v(x)
+|v(x)|^{p(x)-2}v(x)=g \quad\text{in}\quad\mathds{R}^d,
\end{equation}
where $p(x)$ is a measurable function satisfying the estimate $1<p_-\leqslant p(\cdot)\leqslant p_+$ and $A(x)$ is a positive definite matrix  function,
it is natural to introduce the Sobolev spaces with variable exponent $W^{1,p(\cdot)}(\mathds{R}^d)$ defined as the set of measurable functions $v(x)$ such that
$$
\int\limits_{\mathds{R}^d}|v(x)|^{p(x)}dx<+\infty, \qquad \int\limits_{\mathds{R}^d}|\nabla v(x)|^{p(x)}dx<+\infty.
$$
   It was shown in \cite{KoRa,FaZh,Di04} that $W^{1,p(\cdot)}(\mathds{R}^d)$ equipped with the Luxemburg norm is a
  separable Banach space. Moreover, $C_0^\infty(\mathds{R}^d)$ is dense in $W^{1,p(\cdot)}(\mathds{R}^d)$
  if $p(x)$ is $\log$-continuous. The latter result was obtained in \cite{EdRa92,Sa00}  and then improved in \cite{Zhi06}. The Lavrentiev phenomenon was discussed in \cite{Zh1}.

Having this space at hand one can prove that equation  \eqref{perem_diff} is well posed in $W^{1,p(\cdot)}(\mathds{R}^d)$, that is for any $f\in L^{p^*(\cdot)}(\mathds{R}^d)$ equation  \eqref{perem_diff} has
a solution $v\in W^{1,p(\cdot)}$, the solution is unique and satisfies the estimate
$$
\|v\|_{W^{1,p(\cdot)}(\mathds{R}^d)}\leqslant C\|f\|_{L^{p^*(\cdot)}},\qquad
p^*(x):=\frac{p(x)}{p(x)-1}.
$$

Our aim is to realize similar ideas in the above introduced nonlocal setting and to construct
functional spaces corresponding to the functional 
\eqref{intr_functional}
and to the related nonlocal equations. To this end
we consider the functional
$$
F(u)+ \|u\|^{p_-}_{L_{p_-}}= \int\limits_{\Omega\times\Omega}a(x-y) \vp\big(\big|u(y)-u(x)\big|,x,y\big)dxdy+\int\limits_\Omega
|u(x)|^{p_-}dx
$$
and define the  set
\begin{equation}\label{def_L}
\cL(\Omega):=\big\{u\in L_{1,loc}(\Omega)\,:\, f(u)<+\infty\big\},\qquad
f(u):=F(u)+\|u\|^{p_-}_{L_{p_-}}(\Omega).
\end{equation}
We  show that, under natural growth and convexity assumptions on $\vp(\,\cdot\,,x,y)$, this set is linear and, equipped with the Luxemburg norm, it forms a separable Banach space.
Then we explore the properties of this space. In particular, we address the questions of density of
compactly supported infinitely differentiable
functions in this space and characterize the structure of the dual space.

Replacing $F(u)$ with $F_{p(\cdot)}(u)$ in formula \eqref{def_L}, we arrive at the definition of the space $\cL_{p(\cdot)}(\Omega)$ which is a particular case of  the space  $\cL(\Omega)$.
 It is natural to compare  the spaces  $\cL_{p(\cdot)}(\Omega)$,  $p=p(x,y)$,  and the Lebesgue spaces with variable exponent.
Since the functional $F_{p(\cdot)}(u)$ is nonlocal, the variable exponent $p(x,y)$ depends on two arguments $x$ and $y$.
Thus we do not expect that for the generic $p(x,y)$ the space $\cL_{p(\cdot)}(\Omega)$ corresponding to $F_{p(\cdot)}(u)$ coincides
with any Lebesgue space $L_{q(\cdot)}(\Omega)$ with a variable exponent $q(x)$.

Recall that the Sobolev space $W^1_{p(\cdot)}(\Omega)$ with $1<p_-\leqslant p(x)\leqslant p_+$ is embedded in $L_{p(\cdot)}(\Omega)$ and the embedding
is locally compact. In the case of nonlocal functional $F_{p(\cdot)}$ with  $1<p_-\leqslant p(x,y)\leqslant p_+$  the picture is rather different. Although $\cL(\Omega)$ is embedded into $L_{p_-}(\Omega)$,
the embedding is not locally compact any more; here we have $L_{p_-}(\Omega)\cap L_{p_+}(\Omega)\subset \cL(\Omega)\subset L_{p_-}(\Omega)$, see Theorem \ref{th1} below.

As a natural  application  of our results,  one can
%
 develop a qualitative theory of equations
$$
-\int\limits_{\Omega}a(x-y)\vp'(|u(y)-u(x)|,x,y)\frac{u(y)-u(x)}{|u(y)-u(x)|}dy+ |u(x)|^{p_--2}u(x)=g \quad\text{in}\quad\Omega,
$$
where $g\in  \cL^*$,
and $\vp'$ stands for the partial derivative of $\vp$ with respect to the first
variable
   \begin{equation*}
   \vp'(z,x,y)=\frac{\p\vp}{\p z}(z,x,y).
\end{equation*}
This equation is the Euler--Lagrange equation of the functional $F(u)+\|u\|^{p_-}_{L_{p_-}}-\langle g,u\rangle$.
Our results on the structure of the dual space $\cL^*$ imply that the above equation is solvable for any  $g\in  \cL^*$ and the operator on its left hand side maps continuously  $\cL$ on  $\cL^*.$

It should be emphasized that, in connection with the growing number of applications where non-local convolution-type functionals and  operators are used, the study of the corresponding ''nonlocal'' Orlicz-type spaces is becoming a very important task.

The paper is organized as follows. In Section \ref{Sec_2} we provide the detailed list of conditions
on the functional $F(u)$ and formulate our main results.
%
%
%
In Section \ref{sec2} we provide a number of auxiliary technical results.
%
The remaining four sections
are devoted to the proof of 
 main results.
%


\section{Problem setup and main results}\label{Sec_2}

Let $\Om\subseteq\mathds{R}^d$ be an arbitrary   $d$--dimensional domain, where $d\geqslant 1.$ The domain $\Om$ can be both bounded and unbounded; the case $\Om=\mathds{R}^d$ is also admitted.  By $a=a(z),$ $z\in\mathds{R}^d,$ we denote a non--negative function on $\mathds{R}^d,$ we suppose that $a\in L_1(\mathds{R}^d)$ and there exists a fixed non--empty ball $B_0$ centered at zero such that
\begin{equation}\label{2.17}
a(z)\geqslant c_0>0\quad\text{almost everywhere in}\quad B_0.
\end{equation}
If the domain $\Om$ is multi--connected and $\Om_i,$ $i=1,2,\ldots$ are disjoint connected components of $\Om,$ we suppose that the distance between these components is less than the diameter of $B_0:$
\begin{equation}\label{2.19}
\dist(\Om_i,\Om_{i+1})<\diam B_0,\qquad i=1,2,\ldots
\end{equation}

We recall \cite[Ch. 2, Sect. 2.1]{HH} that a real function $r=r(t)$ defined on $\mathds{R}_+$ is called \textit{almost increasing}, respectively, \textit{almost decreasing} if there exists  a constant $\b\geqslant 1$ such that
\begin{equation}\label{2.25}
r(s)\leqslant \b r(t),\qquad\text{respectively,}\qquad \b r(s)\geqslant  r(t),
\end{equation}
for all $0\leqslant s\leqslant t.$  By $\vp=\vp(z,x,y)$ we denote a non--negative real--valued function of the variables  $(z,x,y)\in\mathds{R}_+\times\Om\times\Om,$ which is supposed to obey the following conditions:

\begin{enumerate}\def\theenumi{(C\arabic{enumi})}

\item\label{Meas} For each $u\in L_{1,loc}(\Om)$ the function $\vp(|u(x)-u(y)|,x,y)$ is measurable on $\Om\times\Om.$

\item\label{Conv} The function $\vp(t,x,y)$ is uniformly convex in $t$ for almost all $(x,y),$ namely, for every $\e>0$ there exists $\d\in(0,1)$ such that
    \begin{equation}\label{uniconv}
    \vp\left(\frac{s+t}{2},x,y\right)\leqslant (1-\d) \frac{\vp(s,x,y)+\vp(t,x,y)}{2}
    \end{equation}
 for almost all $(x,y)\in\Om\times\Om$ and all $s>0,$ $t>0$ such that $|s-t|\geqslant \e \max\{s,\,t\}.$

\item\label{Pconv} There exist two constants $p_-$ and $p_+$ such that
$1<p_-\leqslant p_+$ and a constant $\b\geqslant 1$ such that the function $t\mapsto \frac{\vp(t,x,y)}{t^{p_-}}$ is almost increasing with the  constant $\b$ in (\ref{2.25})  for almost all $(x,y)\in\Om\times\Om,$ while the function $t\mapsto \frac{\vp(t,x,y)}{t^{p_+}}$ is almost decreasing with the constant $\b$
for almost all $(x,y)\in\Om\times\Om.$

\item\label{Bound} The function $\vp$
satisfies the relations
\begin{align}\label{2.28}
&c_1^{-1}\leqslant \vp(1,x,y)\leqslant c_1,
\\
&\vp(0,x,y)=0,\qquad \vp(t,x,y)>0,\label{2.29}
\end{align}
for almost all $(x,y)\in\Om\times\Om,$ where $c_1$ is a fixed positive constant independent of $x$ and $y.$

\item\label{Diff} The function $\vp(t,x,y)$ is differentiable in $t>0$ for almost all $(x,y)\in\Om\times\Om$ and satisfies the estimate
    \begin{equation*}
    0< t\vp'(t,x,y)\leqslant c_2\vp(t,x,y),\qquad t>0
    \end{equation*}
    for almost all $(x,y)\in\Om\times\Om$ with a constant $c_2>1$ independent of $t,$ $x,$ and $y.$
\end{enumerate}


We shall show in Lemma~\ref{lm:MonCon} that  Conditions~\ref{Pconv},~\ref{Bound} ensure the identities
\begin{equation}\label{2.30}
\lim\limits_{t\to+0} \vp(t,x,y)=0,\qquad \lim\limits_{t\to+\infty} \vp(t,x,y)=+\infty
\end{equation}
for almost all $(x,y)\in\Om\times\Om.$

On functions $u\in L_{1,loc}(\Om)$ we introduce the functional
\begin{equation}\label{2.2}
F(u):=\int\limits_{\Om\times\Om}
\vp\big(|u(x)-u(y)|,x,y\big)a(x-y)\,dxdy.
\end{equation}
For each $u\in L_{1,loc}(\Om)$ this  functional is either finite or equal to $+\infty$ since the case of a non--integrable integrand is excluded by Condition~\ref{Meas} and the non--negativity of integrand. In what follows we allow $F(u)$ to take the value $+\infty$ and under this assumption this functional is well--defined.

 In terms of this functional we define one more functional
\begin{equation}\label{2.3}
f(u):=|u|_{p,\Om}+\|u\|_{L_{p_-}(\Om)},\qquad |u|_{p,\Om}:=\inf\left\{\l>0:\, F\left(\frac{u}{\l}\right) \leqslant 1\right\}.
\end{equation}
We denote
\begin{equation}\label{2.18}
  \cL(\Om):=\big\{u\in L_{p_-,loc}(\Om):\, f(u)<+\infty\big\}.
\end{equation}

In addition to the space $\cL(\Om),$ we also introduce the space generated  by the functional $F.$ Namely, on the space $L_{1,loc}(\Om)$ we define  the equivalence relation:
$u\sim v$ if  $u-v=const$ almost everywhere on $\Om.$ By $\L(\Om)$ we denote the linear space  consisting  of cosets $\mathbf{U}$  generated by the equivalence relation $\sim$ such that
$|u|_{p,\Om}<\infty$ for some $u\in \mathbf{U},$ and hence for each $u\in \mathbf{U}.$ We extend the functional $|\cdot|_{p,\Om}$ to the space $\L(\Om)$ by the rule
\begin{equation}\label{2.16}
|\mathbf{U}|_{p,\Om}:=|u|_{p,\Om}\quad \text{for}\quad u\in\mathbf{U}.
\end{equation}
It is clear that the above definition is independent of the particular choice of an element $u$ in the coset $\mathbf{U}.$

The main aim of our work is to study the spaces $\cL(\Om)$ and $\L(\Om)$ as well as  their dual spaces.  %

\subsection{Main results}\label{sec_main_res}

Now we are in position to formulate our main results.
The first result is about very basic properties of the space $\cL(\Om).$

\begin{theorem}\label{th1} Assume that Conditions~\ref{Meas},~\ref{Conv},~\ref{Pconv},~\ref{Bound} are satisfied.
Then the functional $\|u\|_{\cL(\Om)}:=f(u)$ is a norm, and the space $\cL(\Om)$ equipped with this norm is Banach.
 The following embeddings are valid
\begin{equation}\label{2.4}
L_{p_+}(\Om)\cap L_{p_-}(\Om)\subseteq \cL(\Om)\subseteq L_{p_-}(\Om).
\end{equation}
The embeddings are continuous.
\end{theorem}

On the space $\cL(\Om)$ we can introduce  equivalent norms. Namely, on $L_{1,loc}(\Om)$ we introduce the functionals
\begin{equation}\label{2.7}
G(u):=F(u)+\int\limits_{\Om}|u(x)|^{p_-}\,dx, \qquad g(u):=\inf\left\{\l>0:\, G\left(\frac{u}{\l}\right) \leqslant 1\right\}.
\end{equation}

\begin{theorem}\label{th4} Assume that Conditions~\ref{Meas},~\ref{Conv},~\ref{Pconv},~\ref{Bound} are satisfied.
Then the functional $g$ is a norm on $\cL(\Om).$  Moreover, this norm is equivalent to $f(u),$ namely,
\begin{equation}\label{2.6}
\frac{1}{2}f(u)\leqslant   g(u) \leqslant \b^{\frac{1}{p_-}}  f(u).
\end{equation}
\end{theorem}

The separability of the space $\cL(\Om)$ and the density of compactly supported infinitely differentiable functions are addressed in the next theorem.

\begin{theorem}\label{th2} Assume that Conditions~\ref{Meas},~\ref{Conv},~\ref{Pconv},~\ref{Bound} are satisfied.
The space of infinitely differentiable compactly supported functions $C_0^\infty(\Om)$ is a dense subset of $\cL(\Om).$ The space $\cL(\Om)$ is separable.
\end{theorem}

The following two theorems establish the main properties of space $\L(\Om)$.

\begin{theorem}\label{th6} Assume that Conditions~\ref{Meas},~\ref{Conv},~\ref{Pconv},~\ref{Bound} are satisfied.
The functional $|\cdot|_{p,\Om}$ is a  norm on the space $\L(\Om).$ The space $\L(\Om)$ with the norm $|\cdot|_{p,\Om}$ is Banach. The norm $|\cdot|_{p,\Om}$ is uniformly convex on $\L(\Om).$
\end{theorem}

\begin{theorem}\label{th5} Let Conditions~\ref{Meas},~\ref{Conv},~\ref{Pconv},~\ref{Bound} be fulfilled,
and assume that  the domain $\Om$ is bounded, connected and has a Lipschitz boundary.
Then $\L(\Om)$ is a Banach space with the norm $|\cdot|_{p,\Om}.$
 Each coset $\mathbf{U}\in \L(\Om)$ contains a unique representative $u,$ which belongs to $L_{p_-}(\Om)$ and has a zero mean value:
 \begin{equation}\label{2.11}
 \la u\ra_\Om=0,\qquad  \la u \ra_\Om:=\frac{1}{\mes\Om}
\int\limits_{\Om} u(x)\,dx.
 \end{equation}
In this sense we say that
\begin{equation}\label{2.8}
\cL(\Om)=\L(\Om)\oplus \mathds{C},
\end{equation}
that is, each function $u\in\cL(\Om)$ is represented as
\begin{equation}\label{2.12}
u=u_{\bot,\Om}+\la u\ra_\Om,\qquad \la u_\bot\ra_\Om=0,
\end{equation}
where $u_{\bot,\Om}$ is an element of the coset in $\L(\Om)$ generated by $u$. For all $u\in \cL(\Om)$  the inequalities
\begin{equation}\label{2.13}
c_3\big(|u_{\bot,\Om}|_{p,\Om}+|\la u\ra_\Om|\big)  \leqslant \|u\|_{\cL(\Om)}\leqslant
c_4\big(|u_{\bot,\Om}|_{p,\Om}+|\la u\ra_\Om|\big)
\end{equation}
hold with positive  constants $c_3$ and $c_4$ independent of $u,$ but depending on $\Om.$
\end{theorem}

We introduce the conjugate function for $\vp$
\begin{equation}\label{2.26}
\vp_*(t):=\sup\limits_{s\geqslant 0} \big(st-\vp(s)\big).
\end{equation}
According to \cite[Ch. 2, Sect. 2.4, Lm. 2.4.2]{HH}, the function $\vp_*(t,x,y)$ is non--negative, convex and   continuous  in $t\geqslant 0$ for almost all $(x,y)\in\Om\times\Om,$ and satisfies Condition~\ref{Meas} and the identities (\ref{2.30}). By \cite[Ch. 2, Sect. 2.4, Prop. 2.4.9]{HH}, the function $\vp_*$ also satisfies Condition~\ref{Pconv} with $p_-$ and $p_+$ replaced by $q_+:=\frac{p_-}{p_--1}$ and $q_-:=\frac{p_+}{p_+-1},$ respectively, that is, the function $\frac{\vp_*(t,x,y)}{t^{q_+}}$ is almost increasing in $t$ and $\frac{\vp_*(t,x,y)}{t^{q_-}}$ is almost decreasing in $t$ for almost all $(x,y)\in\Om\times\Om.$ The function $\vp_*$ also satisfies the Young inequality implied by the definition (\ref{2.26})
\begin{equation}\label{2.27}
st\leqslant \vp_*(t)+\vp(s),\qquad s\geqslant 0,\quad t\geqslant 0.
\end{equation}

On the space $L_{1,loc}(\Om\times\Om)$ we introduce the functionals
\begin{equation}\label{3.8}
\begin{aligned}
&H(U):= \int\limits_{\Om\times\Om} \vp(|U(x,y)|,x,y) a(x-y) \,dxdy,
&& h(U):=\inf\left\{\l>0:\, H\left(\frac{U}{\l}\right)\leqslant 1\right\},
\\
&H_*(U):= \int\limits_{\Om\times\Om} \vp_*(|U(x,y)|,x,y) a(x-y) \,dxdy,
\quad
&&  h_*(U):=\inf\left\{\l>0:\, H_*\left(\frac{U}{\l}\right)\leqslant 1\right\},
\end{aligned}
\end{equation}
which are allowed to take the  infinite value, and the linear sets
\begin{equation}\label{3.9}
\begin{aligned}
&\cH(\Om\times\Om):=\big\{U\in L_{p_-,loc}(\Om\times\Om):\, h(U)<+\infty\big\},
\\
&\cH_*(\Om\times\Om):=\big\{U\in L_{q_-,loc}(\Om\times\Om):\, h_*(U)<+\infty\big\}.
\end{aligned}
\end{equation}

In the space $\cH_*(\Om\times\Om)$ we consider a linear subspace consisting of the functions $W=W(x,y)$ obeying the condition
\begin{equation}\label{2.14}
\int\limits_{\Om} \big(  W(x,y) a(x-y)  -  W(y,x) a(y-x) \big)\,dy=0\quad \text{for a.e.}\quad x\in\Om.
\end{equation}
This subspace is denoted by $\cM.$

Our next results describe the dual spaces for $\L(\Om)$ and $\cL(\Om).$

\begin{theorem}\label{th7} Assume that Conditions~\ref{Meas},~\ref{Conv},~\ref{Pconv},~\ref{Bound},~\ref{Diff} are satisfied.
There exists a one--to--one correspondence between the dual space $(\L(\Om))^*$ and $\L(\Om).$ Namely, for each bounded linear functional $\phi\in (\L(\Om))^*$ there exists a coset $\mathbf{W}\in \L(\Om)$   such that
\begin{align}
&\phi(\mathbf{U})=|\mathbf{W}|_{p,\Om}\frac{\Phi\left(\mathbf{U}, \frac{\mathbf{W}}{|\mathbf{W}|_{p,\Om}}\right)} {\Phi\left(\frac{\mathbf{W}}{|\mathbf{W}|_{p,\Om}},
\frac{\mathbf{W}}{|\mathbf{W}|_{p,\Om}}\right)},\label{2.20}
\\
&\Phi(\mathbf{U},\mathbf{W}):=\int\limits_{\Om\times\Om}\big(u(x)-u(y)\big) \vp'\big(|w(x)-w(y)|,x,y\big)\frac{\overline{w(x)}-\overline{w(y)}}{|w(x)-w(y)|}
 a(x-y)  \,dxdy,\label{2.21}
\end{align}
where $u\in\mathbf{U},$ $w\in\mathbf{W}$ and the above definition of $\Phi$ is independent of the choice of $u$ and $w.$ And vice versa, each coset $\mathbf{W}\in \L(\Om)$  generates a bounded linear functional by the above formula.

For each functional $\phi\in (\L(\Om))^*$   there exists a function $W=W(x,y),$ $W\in \cH_*(\Om\times\Om)$ such that
\begin{equation}\label{2.22}
\phi(u)=\int\limits_{\Om\times\Om}
\big(u(x)-u(y)\big) W(x,y) a(x-y)  \,dxdy.
\end{equation}
Two functions $W_1,\, W_2\in \cH_*(\Om\times\Om)$ generate the same functional by the formula (\ref{2.22}) if and only if $W_1-W_2\in \cM.$
\end{theorem}

It follows from the  definitions of the spaces $\cL(\Om)$  and  $\L(\Om)$ that
the space $\cL(\Om)$ consists of the functions $u\in L_{1,loc}(\Om)$ such that $u \in L_{p_-}(\Om)$ and $\mathbf{U}\in\L(\Om),$ where $\mathbf{U}$ is the coset  generated by the element $u$ and the above introduced relation $\sim.$ In this sense, we can say that
\begin{equation}\label{2.24}
\cL(\Om)=\L(\Om)\cap L_{p_-}(\Om).
\end{equation}

\begin{theorem}\label{th8} Assume that Conditions~\ref{Meas},~\ref{Conv},~\ref{Pconv},~\ref{Bound},~\ref{Diff} are satisfied.
Then each functional $\phi\in(\cL(\Om))^*$ can be represented in the form
\begin{equation}\label{2.23}
\phi(u)=\phi_0(\mathbf{U}) + \int\limits_{\Om} \psi(x) u(x)\,dx,\qquad u\in\cL(\Om),
\end{equation}
where $\phi_0$ is some functional from $(\L(\Om))^*,$ and $\psi$ is some function from $L_{q_-}(\Om),$ while $\mathbf{U}\in \L(\Om)$ is the coset, which includes $u.$
 And vice versa, each pair $(\phi_0,\psi),$ $\phi_0\in(\L(\Om))^*,$ $\psi\in L_{q_-}(\Om)$ generates a linear bounded functional on $(\cL(\Om))^*$ by the formula (\ref{2.23}).
\end{theorem}

\begin{theorem}\label{th3}
Under the assumptions of Theorem~\ref{th5} each bounded linear functional $\phi\in (\cL(\Om))^*$ can be represented in the form
\begin{equation}\label{2.10a}
\phi(u)=\phi_0(u)+k \la u\ra_{\Om},
\end{equation}
where $\phi_0\in(\L(\Om))^*$ is some functional and $k\in\mathds{C}$ is some constant. And vice versa, each pair $(\phi_0,k),$ $\phi_0\in(\L(\Om))^*,$ $k\in\mathds{C},$
 generates a bounded linear functional by the above formula.
\end{theorem}

The natural question is how large the class of   functions satisfying Conditions~\ref{Meas}--\ref{Diff} and whether it is invariant under some operations.
 The answer 
 is given by the next  statement.
\begin{theorem}\label{l_sum_prod}
Let  functions $\vp(z,x,y)$ and $\psi(z,x,y)$ satisfy Conditions \ref{Meas}--\ref{Diff}. Then the following functions also satisfy the same conditions:
\begin{align*}
&\tht(z,x,y):=\vp(z,x,y)+\psi(z,x,y),
\\
&\tht(z,x,y):=\vp(z,x,y)b(x,y),\qquad b\in L_\infty(\Om\times\Om),\quad 0<c_5<b(x,y)<c_6,\quad c_5,c_6=const,
\\
&\tht(z,x,y):=\vp(z,x,y)\psi(z,x,y),
\\
&\tht(z,x,y):=\vp(\psi(z,x,y),x,y).
\end{align*}
\end{theorem}

 Since condition \ref{Conv} is difficult to verify, we formulate below a sufficient condition for it to hold.
 \begin{lemma}\label{lem_sec_deriv}
 Let a function $\vp(z,x,y)$ satisfy conditions \ref{Meas} and \ref{Pconv}--\ref{Diff}. Assume that the function
 $\vp(z,x,y)$
 is twice continuously differentiable in $z$ for almost all $(x,y)\in\Om\times\Om$ and satisfies the estimate
\begin{equation}\label{secderbou}
  \vp''(z,x,y)\geqslant c_7z^{-2}\vp (z,x,y), \qquad z>0,
\end{equation}
with a constant $c_7$ independent of $z,$ $x,$ and $y$. Then the function $\vp(z,x,y)$ satisfies Condition \ref{Conv}.
 \end{lemma}

With the help of this statement we show that Conditions \ref{Conv}--\ref{Diff} are stable under relatively small general perturbations.

\begin{theorem}\label{cor_pert}
Let a function $\vp(z,x,y)$ satisfy the  Conditions \ref{Meas}--\ref{Diff} and inequality (\ref{secderbou}). Assume that
  a function $\psi(z,x,y)$
 is twice continuously differentiable in $z$   for almost all $(x,y)\in\Om\times\Om$ and satisfies Condition~\ref{Meas} and the relations
 \begin{align}\label{2.32}
& \psi(0,x,y)=0,\qquad \psi'(0,x,y)=0 && \text{for almost all}\quad (x,y)\in\Om\times\Om,
\\
&\label{2.31}
\big|\psi''(z,x,y)\big|\leqslant c_8 \vp''(z,x,y) &&
 \text{for almost all}\quad (x,y)\in\Om\times\Om \quad\text{and all}\quad z>0,
 \end{align}
where $c_8<1$ is a constant independent of $z,$ $x,$ and $y.$
Then the sum $\vp(z,x,y)+\psi(z,x,y)$ satisfies the Conditions \ref{Meas}--\ref{Diff}.
\end{theorem}

\begin{theorem}\label{th2.11}
Let a function $\vp(z,x,y)$ satisfy Conditions \ref{Meas}--\ref{Diff} and inequality (\ref{secderbou}). Assume that
  a function $\psi(z,x,y)$
 is non--negative, non--decreasing, and twice continuously differentiable in $z$   for almost all $(x,y)\in\Om\times\Om,$   satisfies the Conditions~\ref{Meas},~\ref{Bound},~\ref{Diff} and the estimate
\begin{equation}\label{2.33}
 \psi''(z,x,y)\geqslant -c_9 z^{-1}\psi'(z,x,y)-c_{10} z^{-2} \psi(z,x,y)\quad  \text{for almost all}\quad (x,y)\in\Om\times\Om,
 \end{equation}
where $c_9<2$, $c_{10}<c_7$ are some fixed constant independent of $z,$ $x,$ $y$ and $c_7$ is the constant from the inequality (\ref{secderbou}) for $\vp.$ Suppose also that the function $z\mapsto \frac{\psi(z,x,y)}{z^q}$ is non--increasing for almost all $(x,y)\in\Om\times\Om$ with some fixed $q\geqslant 0$ independent of $z,$ $x,$ $y.$
Then the function $\psi(z,x,y)\vp(z,x,y)$ satisfy the Conditions~\ref{Meas}--\ref{Diff}.
\end{theorem}


\subsection{Discussion of main results}

Let us briefly discuss the main features of our problem and main results. The Orlicz spaces $\cL$ and $\L$ are defined on the base of functional $F,$ which has a non--local convolution like structure  due to the presence of functions $a(x-y)$ and   $u(x)-u(y)$
in (\ref{intr_functional}). This makes an essential difference from the classical Orlicz spaces, the definition of which is based on the  functionals
\begin{equation*}
u\mapsto \int\limits_{\Om} \vp(|u(x)|,x)\,dx.
\end{equation*}
Moreover, while the integrand in the latter functional should be strictly positive, this is not the case for the integrand in (\ref{intr_functional}). The integrand in (\ref{intr_functional}) is to be just non--negative
since the function $a=a(z)$ is allowed to vanish, see the condition (\ref{2.17}). In particular, the function $a$ can be  compactly supported.

We also observe that   the integral (\ref{intr_functional}) uses only the values of function $a(z)$ for
\begin{equation*}
z\in\Om_\natural:=\{x-y:\, x,y\in\Om\}
\end{equation*}
and for $z\in\mathds{R}^d\setminus\Om_\natural$ the values $a(z)$ play no role. For instance,  we can suppose that $a$ vanishes identically on $\mathds{R}^d\setminus\Om_\natural$ and then we just should  impose the condition $a\in L_1(\Om_\natural).$ Under such assumption, if the domain $\Om$ is bounded, we can take $a$ obeying the lower bound (\ref{2.17}) on $\Om_\natural,$ for instance, $a$ can be identically constant on this domain. At the same time, if $\Om$ is unbounded, then $a$ can not satisfy the lower bound (\ref{2.17}) for all $z\in\Om_\natural$ since this contradicts the assumption $a\in L_1(\Om_\natural).$

It should be also emphasized that the interplay between the functions $\vp(z,x,y)$ and $a(x-y)$ can affect essentially
the structure of  the space $\cL$. We illustrate this with the following example. Let
\begin{equation*}
\vp(z,x,y)=z^{p(x,y)}, 
\end{equation*}
where  $p(x,y)$ satisfies the estimates
$$
p(x,y)<2\quad\text{for}\quad |x-y|\leqslant 1, \qquad p(x,y)>3\quad\text{for}\quad |x-y|\geqslant 3.
$$
We let  $a_1=\chi_{\{z:\, |z|\leqslant 1\}}$, $a_2(z)=\chi_{\{z:\, |z|\leqslant 5\}}$, where $\chi_\om$ is the characteristic function of a set $\om.$ We
 consider the corresponding Banach spaces $\cL(\mathds{R}^d),$ which we denote by $\cL^{(1)}(\mathds{R}^d)$ and $\cL^{(2)}(\mathds{R}^d)$.
Then by direct inspection we see that the space $L_2(\mathds{R}^d)\cap L_1(\mathds{R}^d)$ is embedded into $\cL^{(1)}(\mathds{R}^d)$
but is not embedded into $\cL^{(2)}(\mathds{R}^d)$.

Conditions~\ref{Meas},~\ref{Conv},~\ref{Pconv},~\ref{Bound} are standard and natural, exactly the same conditions are usually imposed in the classical theory of Orlicz spaces, see \cite{HH}. These conditions are sufficient to state the main properties of the spaces $\cL(\Om),$ $\L(\Om)$ in Theorems~\ref{th1},~\ref{th4},~\ref{th2},~\ref{th6},~\ref{th5}, which characterize their structure. In order to describe the dual spaces, we additionally need  Condition~\ref{Diff}. This condition is essentially employed in the proofs of Theorems~\ref{th7},~\ref{th8},~\ref{th3}. The problem of describing the dual spaces $\cL^*(\Om)$ and $\L^*(\Om)$ becomes quite complicated for the  functions $\phi$, which do not obey Condition~\ref{Diff}. Moreover, the structure of the dual spaces in this case likely differs drastically from the situations described in Theorems~\ref{th7},~\ref{th8},~\ref{th3}.

The general theory of Orlicz spaces $\cL(\Om)$ and $\L(\Om)$ is to be supported by a reasonable series of  appropriate examples of the function $\vp(z,x,y)$. Theorems~\ref{l_sum_prod},~\ref{cor_pert},~\ref{th2.11} describe the structure of class of admissible functions $\vp.$ According to Theorem~\ref{l_sum_prod}, for any collection of admissible functions
$\vp_1,\ldots,\vp_N$ and any  essentially bounded uniformly positive functions $b_1(x,y),\ldots,b_N(x,y)$ the linear
combination $b_1(x,y)\vp_1(z,x,y)+\ldots+b_N(x,y)\vp_N(z,x,y)$ is also an admissible function.
This class is also closed with respect to the usual multiplication and to the composition. Lemma~\ref{lem_sec_deriv} provides an effective way to verify the Condition~\ref{Conv}, namely, once the function $\vp$ is   twice differentiable in $z$ and  satisfies the estimate (\ref{secderbou}), this ensures Condition~\ref{Conv}. Although
the estimate \eqref{secderbou} is a sufficient condition of validity of  Condition~\ref{Conv},
 it holds for a wide class of admissible functions $\vp$. It turns out that the sufficient condition given by Lemma~\ref{lem_sec_deriv} is quite close to be necessary. Theorems~\ref{cor_pert},~\ref{th2.11} say that once the function $\vp$ satisfies the Conditions~\ref{Meas}--\ref{Diff}, we can add to $\vp$ and multiply it by a function $\psi,$ which even does not obey the same conditions. This essentially enlarges the class of admissible functions $\vp.$

Let us  dwell on some examples of the function $\vp.$
A very important example, which served as the main motivation of present work, is
\begin{equation*}
\vp(z,x,y)=z^{p(x,y)}b(x,y),
\end{equation*}
where $p$ and
$b$ are bounded  positive measurable functions and $p(x,y)$  obeys the inequality $1<p_-\leqslant p(\cdot)\leqslant p_+$ with some fixed constants $p_-$ and $p_+.$ It is easy to verify that Conditions~\ref{Meas}, \ref{Conv}, \ref{Pconv}, \ref{Bound}, \ref{Diff} are satisfied by this function $\vp.$

A more general example is given by
\begin{equation*}
\vp(z,x,y)=\sum\limits_{i=1}^{m} z^{p_i(x,y)}b_i(x,y),
\end{equation*}
where  $p_i$ and
$b_i$ are bounded  positive measurable functions and each $p_i(x,y)$  obeys the inequality $1<p_-\leqslant p(\cdot)\leqslant p_+$ with some fixed constants $p_-$ and $p_+$ independent of $i.$  According to Theorem \ref{l_sum_prod},
for such a function $\vp(z,x,y)$ Conditions \ref{Meas}--\ref{Diff} also hold. Furthermore, if $P(\xi_1,\ldots, \xi_m)$ is a polynomial function with positive coefficients, and conditions
\ref{Meas}--\ref{Diff} hold for functions
$\vp_1(z,x,y),\ldots,\vp_m(z,x,y)$, then these conditions also hold for the function   $P(\vp_1(z,x,y),\ldots, \vp_m(z,x,y))$.


The function $\psi$ in the formulation of Theorem
\ref{cor_pert} need not be convex. A typical example of a function obeying the conditions of this corollary reads
$$
\vp(z,x,y):=10 B(x,y) z^2,\qquad \psi(z,x,y):= B(x,y) \sin^3 z,     \qquad 0<B_-\leqslant B(x,y)\leqslant B_+,
$$
where $B_\pm$ are some fixed constants independent of $z,$ $x,$ $y$ and $B\in L_\infty(\Om\times\Om).$ An example of the function $\psi$ in Theorem~\ref{th2.11}
is
\begin{align*}
&\psi(z,x,y):=\ln^{\g(x,y)}(1+\Ups(x,y)z),\qquad \g,\Ups\in L_\infty(\Om\times\Om), \\
&  1\leqslant\g_-\leqslant\g(x,y)\leqslant \g_+,\hphantom{+\Ups(x,y)z),}\qquad 0<\Ups_-\leqslant\Ups(x,y)\leqslant \Ups_+,
\end{align*}
where $\g_\pm$ and $\Ups_\pm$ are some constants independent of $z,$ $x,$ and $y.$ It is straightforward to verify that the function $z\mapsto \frac{\psi(z,x,y)}{z^\g_+}$ has a negative derivative and hence, is non--increasing. The estimate (\ref{2.33}) can be also verified by direct computations; it turns out that we can take $c_{10}=0$ and $c_9=1.$ This is why the above function $\psi$ can be multiplied by each function $\vp$ obeying the Conditions~\ref{Meas}--\ref{Diff} and the inequality~(\ref{secderbou}) with an arbitrary $c_7.$ The further examples of functions $\psi$ obeying the assumptions of Theorems~\ref{cor_pert},~\ref{th2.11} are also possible. This is why the class of admissible functions $\vp$ is indeed very large and this makes the theory of the Orlicz spaces $\cL(\Om)$ and $\L(\Om)$ very rich.

It should be also stressed that
 Conditions~\ref{Conv},   \ref{Pconv}, \ref{Bound} imply that the function $\vp(z,x,y)$ has a polynomial growth in the variable $z$, see the estimate~(\ref{3.28}) in Lemma~\ref{lm:est}. This property of the function
 $\varphi$ is essentially used in the proof of our results.
For the functions $\vp$ with a faster growth in $z$ the properties of the functional
$F(u)$ defined in \eqref{intr_functional} might change drastically. In particular, the domain of $F$ need not be a linear set any more and the function
$\l\mapsto F\big(\frac {u}{\l}\big)$ need not be continuous in $\l$ for each $u$ from the domain of  $F$.
Although the normed space $\cL$ can still be defined,  the studying of its properties requires new arguments.
This is an interesting open problem.

 Using our main results, one can show that, for each $g\in \cL^*$,  Equation \eqref{perem_diff} has a unique
solution  $u\in \cL$, and the operator on the right-hand side of \eqref{perem_diff} maps $\cL$ onto  $\cL^*$.
The functional $F(u)+\|u\|^{p_-}_{L_{p_-}}- \langle g,u\rangle$  attains its minimum in $\cL$ at a function, which is unique even in $L_{1,loc}(\Om).$

\section{Preliminaries}\label{sec2}

Before proceeding to the proof of main results, we discuss some simple properties of the functionals $F,$   $G,$ and the function $\vp,$ which will be employed throughout the work.

We begin with a simple lemma.

\begin{lemma}\label{lm:est} Under Conditions~\ref{Conv},~\ref{Pconv},~\ref{Bound}
the function $\vp$ satisfies the estimates
\begin{align}
&\label{3.22a}
\b^{-1} \min\{\l^{p_-},\l^{p_+}\} \vp\left(\frac{t}{\l},x,y\right)\leqslant \vp(t,x,y)\leqslant \b \max\{\l^{p_-},\l^{p_+}\} \vp\left(\frac{t}{\l},x,y\right),
\\
&\label{3.28}
\b^{-1} c_1^{-1} \min\big\{t^{p_+},\,t^{p_-}\big\}\leqslant \vp(t,x,y)\leqslant \b c_1 \max\big\{t^{p_+},\,t^{p_-}\big\},
\\
&\label{3.30}
\vp(t+s,x,y)\leqslant \frac{\vp(2t,x,y)+\vp(2s,x,y)}{2} \leqslant 2^{p_+-1}\b
\big(\vp(t,x,y)+\vp(s,x,y)\big)
\end{align}
for all $t>0,$ $s>0,$ $\l>0.$ If, in addition, Condition~\ref{Diff} holds, then
\begin{equation} \label{3.31}
 \vp_*(\vp'(t))\leqslant (c_2-1)\vp(t)
\end{equation}
for all $t>0.$
\end{lemma}

\begin{proof}
By Condition~\ref{Pconv}, the functions
\begin{equation*}
t\mapsto \frac{\vp(t,x,y)}{t^{p_-}},\qquad t\mapsto \frac{\vp(t,x,y)}{t^{p_+}}
\end{equation*}
are respectively almost increasing and almost decreasing; without loss of generality we suppose that they satisfy the inequality (\ref{2.25}) with the same constant $\b.$
We take an arbitrary $\l\in(0,1)$ and write the inequalities in (\ref{2.25}) for
\begin{equation}\label{3.23}
r(t)=\frac{\vp(t,x,y)}{t^{p_-}}\qquad\text{and}\qquad r(t)=\frac{\vp(t,x,y)}{t^{p_+}}
\end{equation}
with $s$ and $t$ replaced by $t$ and $\frac{t}{\l}.$
This gives
\begin{equation}\label{3.22}
\b^{-1} \l^{p_+} \vp\left(\frac{t}{\l},x,y\right)\leqslant \vp(t,x,y)\leqslant \b \l^{p_-} \vp\left(\frac{t}{\l},x,y\right).
\end{equation}
Then we take an arbitrary $\l\in(1,+\infty)$ and we write the inequalities in (\ref{2.25}) for the functions (\ref{3.23}) with $s$ and $t$ replaced by $\frac{t}{\l}$ and $t$
\begin{equation}\label{3.24}
\b^{-1} \l^{p_-} \vp\left(\frac{t}{\l},x,y\right)\leqslant \vp(t,x,y)\leqslant \b \l^{p_+} \vp\left(\frac{t}{\l},x,y\right).
\end{equation}
This estimate and (\ref{3.22}) imply (\ref{3.22a}).

We write inequalities (\ref{2.25}) for the functions (\ref{3.23}) for $s=1$ and $t>0.$ This gives
\begin{equation*}
\b^{-1} \vp(1,x,y) \min\big\{t^{p_+},\,t^{p_-}\big\}\leqslant \vp(t,x,y)\leqslant \b \vp(1,x,y) \max\big\{t^{p_+},\,t^{p_-}\big\}.
\end{equation*}
Using then the inequalities (\ref{2.28}) from Condition~\ref{Bound}, we arrive at (\ref{3.28}).
It follows from the second inequality in (\ref{2.25}) for the second function in (\ref{3.23}) with $s=2t$ that
\begin{equation}\label{3.29}
\vp(2t,x,y)\leqslant 2^{p_+} \b \vp(t,x,y),\qquad t>0.
\end{equation}
Then by the convexity of function $\vp,$ see Condition~\ref{Conv}, for all $t>0,$ $s>0$
we have
\begin{equation*}
\vp(t+s,x,y)\leqslant \frac{\vp(2t,x,y)+\vp(2s,x,y)}{2} \leqslant 2^{p_+-1}\b
\big(\vp(t,x,y)+\vp(s,x,y)\big).
\end{equation*}
and this proves (\ref{3.30}).

The convexity of function $\vp$
implies that the derivative $\vp'(t,x,y)$ is monotonically non--decreasing. Then by the Lagrange's mean  value theorem  we have
\begin{equation*}
\frac{\vp(t)-\vp(s)}{t-s}=\vp'(\vt_-)\leqslant \vp'(t),\quad 0<s<t,\qquad  \frac{\vp(s)-\vp(t)}{s-t}=\vp'(\vt_+)\geqslant \vp'(t),\quad 0<t<s,
\end{equation*}
where $\vt_-\in(s,t)$ and $\vt_+\in(t,s)$ are some points. We rewrite the obtained inequality as
\begin{equation*}
s \vp'(t)-\vp(s)\leqslant t \vp'(t)-\vp(t),
\end{equation*}
and in view of   the definition (\ref{2.26}) of function $\vp_*$ we find
\begin{equation*}
\vp_*(\vp'(t))\leqslant t \vp'(t)-\vp(t).
\end{equation*}
Using now Condition~\ref{Diff}, we obtain (\ref{3.31}). The proof is complete.
\end{proof}

\begin{lemma}\label{lm:MonCon} Under Conditions~\ref{Conv},~\ref{Pconv},~\ref{Bound}
the function $\vp(t,x,y)$ is convex, strictly  monotonically increasing, differentiable  and continuous in $t\in[0,+\infty)$ for almost all $(x,y)\in\Om\times\Om.$ The identities (\ref{2.30}) hold.
\end{lemma}

\begin{proof}
The convexity of function $\vp$ is postulated in Condition~\ref{Conv}. Using the convexity of $\vp $ and  the relations (\ref{2.29}), for arbitrary $0<s<t$ we have
\begin{equation*}
\vp(s,x,y)\leqslant \frac{s}{t}\vp(t,x,y)+\left(1-\frac{s}{t}\right) \vp(0,x,y) \leqslant
\frac{s}{t}\vp(t,x,y)<\vp(t,x,y)
\end{equation*}
for almost all $(x,y)\in\Om\times\Om,$ which implies the required monotonicity and the first identity in (\ref{2.30}). By the convexity and (\ref{2.28}), (\ref{2.29}) for all $t>1$ we also have
\begin{equation*}
c_1^{-1}\leqslant \vp(1,x,y)\leqslant \frac{1}{t}\vp(t,x,y)+\left(1-\frac{1}{t}\right) \vp(0,x,y)=\frac{1}{t}\vp(t,x,y),\qquad t c_1^{-1}\leqslant \vp(t,x,y)
\end{equation*}
for almost all $(x,y)\in\Om\times\Om.$ This proves the second identity in (\ref{2.30}).

The monotonicity and finiteness of the function $\vp(t,x,y)$ for $t\geqslant 0$ and almost all $(x,y)\in\Om\times\Om$ implies that for each $z>0$ the one--sided limits $\lim\limits_{t\to z\pm 0} \vp(t,x,y)$ are well--defined and obey the inequality
\begin{equation}\label{3.32}
\lim\limits_{t\to z- 0} \vp(t,x,y)\leqslant \lim\limits_{t\to z+0} \vp(t,x,y)
\end{equation}
for almost all $(x,y)\in\Om\times\Om.$ We choose $s<z<\tau<t$ and by the convexity of $\vp$ we obtain
\begin{equation*}
\vp(\tau,x,y)\leqslant \frac{t-\tau}{t-s} \vp(s,x,y)+\frac{\tau-s}{t-s}\vp(t,x,y)
\end{equation*}
for almost all $(x,y)\in\Om\times\Om.$ In this inequality we pass to the limit as $\tau\to z+0$ and then as $s\to z-0$  that gives
\begin{equation*}
 \lim\limits_{t\to z+0} \vp(t,x,y)\leqslant \lim\limits_{t\to z- 0} \vp(t,x,y).
\end{equation*}
Together with (\ref{3.32}) this inequality means that the function $\vp(t,x,y)$ is continuous in $t>0$ for almost all $(x,y)\in\Om\times\Om.$ The proof is complete.
%
 \end{proof}

Assume that Conditions~\ref{Meas},~\ref{Conv},~\ref{Pconv},~\ref{Bound} are satisfied.
It follows from the estimate (\ref{3.22a}) that if for some function $u\in L_{1,loc}(\Om)$ and some $\l_0>0$ the quantity $F(\frac{u}{\l_0})$ is finite, then $F(\frac{u}{\l})$ is finite for each $\l>0.$ In particular, in view of the definition~(\ref{2.18}) of space $\cL(\Om),$ the functionals $f$ and $|\cdot|_{p,\Om}$ are finite on each element of this space. Similarly, for each $\mathbf{U}\in\L(\Om)$ and each $u\in\mathbf{U}$ the quantity $|\mathbf{U}|_{p,\Om}=|u|_{p,\Om}$ is also finite.

It follows from the monotonicity of  function $\vp$ proven in Lemma~\ref{lm:MonCon} that for each $\mathbf{U}\in\L(\Om)$ and each $u\in\mathbf{U}$ the functional $F\left(\frac{u}{\l}\right)$ is strictly  monotonically decreasing in $\l>0.$  Moreover, Lemma~\ref{lm:MonCon} implies that the functional $F\left(\frac{u}{\l}\right)$ is also continuous in $\l>0.$ Indeed, let this functional be finite for $\l\in[\l_1,\l_2]$ for some $\l_1,\, \l_2>0.$ Then in view of the monotonicity of $\vp$  we have the upper bound
\begin{equation}\label{3.21}
\vp\left(\frac{|u(x)-u(y)|}{\l},x,y\right)\leqslant \vp\left(\frac{|u(x)-u(y)|}{\l_1},x,y\right)\quad\text{for almost all}\quad (x,y)\in\Om\times\Om.
\end{equation}
The continuity of $\vp,$ see Lemma~\ref{lm:MonCon}, yields that, as $\l\to\l_0,$ $\l,\,\l_0\in[\l_1,\l_2],$ we have the convergence
\begin{equation*}
\vp\left(\frac{|u(x)-u(y)|}{\l},x,y\right)\to \vp\left(\frac{|u(x)-u(y)|}{\l_0},x,y\right)\quad
\text{for almost all}\quad (x,y)\in\Om\times\Om.
\end{equation*}
Hence, in view of the upper bound (\ref{3.21}), the Lebesgue theorem on dominated convergence implies the continuity of the functional $F\left(\frac{u}{\l}\right)$ in $\l>0.$

The established monotonicity and continuity of  functional $F\left(\frac{u}{\l}\right)$ yield that  the value $|u|_{p,\Om}$ is the unique solution of equation
\begin{equation}\label{3.16}
F\left(\frac{u}{\l}\right)=1.
\end{equation}
The same can be easily established for the functional $G,$ namely, $g(u)$ is the unique solution of the equation
\begin{equation}\label{3.16a}
G\left(\frac{u}{\l}\right)=1,
\end{equation}
and the left hand side of this equation is continuous and monotone in $\l>0.$

In (\ref{3.22a}) we let $\l=|u|_{p,\Om},$ $t=|u(x)-u(y)|,$ multiply the    inequality by $a(x-y)$ and integrate the result over $\Om\times\Om.$ In view of Equation~(\ref{3.16}) we obtain
\begin{equation}\label{3.17a}
 \b^{-1}\min\big\{|u|_{p,\Om}^{p_-},|u|_{p,\Om}^{p_+}\big\}\leqslant F(u)\leqslant \b\max\big\{|u|_{p,\Om}^{p_-},|u|_{p,\Om}^{p_+}\big\}.
\end{equation}
In the same way for $U\in\cH(\Om\times\Om)$ we get
\begin{equation}\label{3.17b}
 \b^{-1}\min\big\{h^{p_-}(U), h^{p_+}(U)\big\}\leqslant H(U)\leqslant \b\max\big\{h^{p_-}(U), h^{p_+}(U)\big\}.
\end{equation}
Since the function $\vp_*$ satisfies Condition~\ref{Pconv} with $q_+$ and $q_-,$ as above we establish the estimate
\begin{equation}\label{3.17c}
 \b^{-1}\min\big\{h_*^{q_-}(U), h_*^{q_+}(U)\big\}\leqslant H_*(U)\leqslant \b\max\big\{h_*^{q_-}(U), h_*^{q_+}(U)\big\}
\end{equation}
 for $U\in\cH_*(\Om\times\Om).$

Assume that $\Om$ is a bounded domain
and $\mathbf{U}\in\L(\Om)$ is some coset. Then for each $u\in\mathbf{U}$
by (\ref{3.16}) we have
\begin{equation}\label{3.19}
F(\tilde{u})=1,\qquad \tilde{u}:=\frac{u}{|u|_{p,\Om}},\qquad |\tilde{u}|_{p,\Om}=1.
\end{equation}
Letting
\begin{equation*}
\Pi_+:=\big\{(x,y)\in\Om\times\Om:\, |\tilde{u}(x)-\tilde{u}(y)|\geqslant 1\big\},\qquad
\Pi_-:=\big\{(x,y)\in\Om\times\Om:\, |\tilde{u}(x)-\tilde{u}(y)|<1\big\},
\end{equation*}
and considering (\ref{3.19}), the left inequality in (\ref{3.28}) and the boundedness of $\Om$ we obtain
\begin{align*}
\Bigg(\int\limits_{\Pi_+}&+\int\limits_{\Pi_-} \Bigg) |\tilde{u}(x)-\tilde{u}(y)|a(x-y)\,dxdy
\\
\leqslant  & \int\limits_{\Pi_+} |\tilde{u}(x)-\tilde{u}(y)|^{p_-}a(x-y) \,dxdy
+\int\limits_{\Pi_-} a(x-y)\,dxdy
\\
\leqslant& \b c_1\int\limits_{\Pi_+} \vp\big(|\tilde{u}(x)-\tilde{u}(y)|,x,y\big) a(x-y) \,dxdy+C  \leqslant\b c_1+ C,
\end{align*}
where $C$ is some fixed constant independent of $u$ and $\tilde{u}.$
Returning back to the function $u,$ we   easily find
\begin{equation}\label{3.20a}
\int\limits_{\Om\times\Om} |u(x)-u(y)|a(x-y)\,dxdy
 \leqslant C |u|_{p,\Om}
\end{equation}
with some constant $C$ independent of $u.$

\section{Banach space and equivalent norm}\label{sec3}

In this section we prove Theorems~\ref{th1},~\ref{th4}.

\begin{proof}[Proof of Theorem~\ref{th1}]
Let us verify that the functional $f$ is a norm on $\cL(\Om).$ If $f(u)=0,$ it follows from the definition~(\ref{2.3}) that $\|u\|_{L_{p_-}(\Om)}=0$ and hence, $u=0$ almost everywhere in $\Om.$ The homogeneity property and triangle inequality for $|\cdot|_{p,\Om}$ can be proved by literal reproducing the proof of  Lemma~3.2.2 in \cite[Ch. 3, Sect. 3.2]{HH}. Hence,  the functional $f$ is a norm on $\cL(\Om).$ In what follows we redenote this functional by $\|\cdot\|_{\cL(\Om)}.$

We turn to proving that the space $\cL(\Om)$ is Banach. Let $u\in \cL(\Om)$ be a fundamental sequence in this space, that is,
\begin{equation}\label{3.2}
\|u_n-u_m\|_{\cL(\Om)}\to 0,\qquad n,m\to+\infty.
\end{equation}
This immediately implies that the sequence $u_n$ is fundamental in $L_{p_-}(\Om)$ and is bounded in $\cL(\Om)$:
\begin{equation}\label{3.3}
\|u_n\|_{\cL(\Om)}\leqslant C,\qquad n\in\mathds{N},\qquad \|u_n-u_m\|_{L_{p_-}(\Om)}\to0,\qquad n,m\to+\infty,
\end{equation}
where $C$ is some constant independent of $n.$ Since the sequence $u_n$ is fundamental in the space $L_{p_-}(\Om)$, it converges in $L_{p_-}(\Om)$ to some function $u\in L_{p_-}(\Om).$ The latter convergence implies that there exists a subsequence of $u_n,$ again denoted by $u_n,$ which converges to $u$ almost everywhere in $\Om$ and hence, the sequence
$U_n(x,y):=u_n(x)-u_n(y)$ converges to $U(x,y):=u(x)-u(y)$ almost everywhere in $\Om\times\Om.$ This yields that, as $n\to+\infty,$
\begin{align}\label{3.4}
& \vp\big(|U_n(x,y)-U_m(x,y)|,x,y\big) \to \vp\big(|U(x,y)-U_m(x,y)|,x,y\big), \\
& \vp\big(|U_n(x,y)|,x,y\big)\to
 \vp\big(|U(x,y)|,x,y\big)
\label{3.5}
\end{align}
almost everywhere in $\Om\times\Om$ for each $m\in\mathds{N}.$

The definition  of the norm $\|\cdot\|_{\cL(\Om)}$ and the first bound in (\ref{3.3}) yield that there exists a sequence of numbers $\l_n>0$ such that
\begin{equation}\label{3.6}
F\left(\frac{u_n}{\l_n}\right)\leqslant 1,\qquad \l_n\leqslant C.
\end{equation}
By (\ref{3.17a}) we  then have
\begin{equation}\label{3.7}
F(u_n)\leqslant \b\max\{C^{p_-}, C^{p_+}\}.
\end{equation}
This estimate, convergence (\ref{3.5}) and Fatou lemma yield
\begin{equation*}
F(u)\leqslant\b\max\{C^{p_-}, C^{p_+}\}
\end{equation*}
and hence, $u$ is an element of the space $\cL(\Om).$

Proceedings as in (\ref{3.6}), (\ref{3.7}), by (\ref{3.2}) we find that for each $\e>0$ there exists $N=N(\e)>0$ such that as $n,m>N(\e)$ we have $F(u_n-u_m)<\e$. Using then convergence (\ref{3.4}) and Fatou lemma and passing to the limit as $n\to+\infty,$ we get the bound $F(u_m-u)<\e$ for all $m>N(\e).$ Then the definition of the norm $\|\cdot\|_{\cL(\Om)}$ yields that $\|u_m-u\|_{\cL(\Om)}<\e$ and this proves that the space $\cL(\Om)$ is Banach.

The definition of the functional $f$ implies immediately that $\cL(\Om)$ is a subspace of $L_{p_-}(\Om)$ and this embedding is continuous.
 It follows from   (\ref{3.30}) that
\begin{equation}\label{3.25}
\vp\big(|u(x)-u(y)|,x,y\big)\leqslant 2^{p_+-1}\b\Big( \vp\big(|u(x)|,x,y\big)+\vp\big(|u(y)|,x,y\big)\Big).
\end{equation}

We partition the domain $\Om$  as
\begin{equation*}
\Om=\Om_u^-\cup\Om_u^+,\qquad \Om_u^-:=\{x\in\Om:\, |u|\leqslant 1\},
\qquad \Om_u^+:=\{x\in\Om:\, |u|>1\}.
\end{equation*}
For $x\in\Om_u^-$ by Condition~\ref{Pconv} and the first inequality in (\ref{2.25}) with $s=1$ and $t=|u(x)|$ we find
\begin{equation*}
\vp\big(|u(x)|,x,y\big)\leqslant \b\vp(1,x,y) |u(x)|^{p_-}.
\end{equation*}
For $x\in\Om_u^+$ by Condition~\ref{Pconv} and the second inequality in (\ref{2.25}) with $s=|u(x)|$ and $t=1$  we find
\begin{equation*}
\vp\big(|u(x)|,x,y\big)\leqslant \b \vp(1,x,y) |u(x)|^{p_+}.
\end{equation*}
Two above inequalities and (\ref{3.28}) yield
\begin{equation}\label{3.26}
\vp\big(|u(x)|,x,y\big)\leqslant \b c_1\max\big\{
|u(x)|^{p_-},\, |u(x)|^{p_+}\big\}.
\end{equation}
In the same way one can verify that
\begin{equation}\label{3.27}
\vp\big(|u(y)|,x,y\big)\leqslant \b c_1 \max\big\{
|u(y)|^{p_-},\, |u(y)|^{p_+}\big\}.
\end{equation}
Two obtained inequalities allow us to continue estimating in (\ref{3.25})
\begin{equation*}
\vp\big(|u(x)-u(y)|,x,y\big)\leqslant 2^{p_+-1}\b^2 c_1 \Big(
|u(x)|^{p_-}+|u(x)|^{p_+} +
|u(y)|^{p_-}+|u(y)|^{p_+}
\Big).
\end{equation*}
For each $u\in L_{p_+}(\Om)\cap L_{p_-}(\Om)$
this inequality implies
\begin{align*}
F(u)
\leqslant  &
2^{p_+-1}\b^2 c_1
  \int\limits_{\Om\times\Om} \big(|u(x)|^{p_-}+ |u(x)|^{p_+})(a(x-y)+a(y-x)) \, dx dy
\\
\leqslant&2^{p_+}\b^2 c_1
\|a\|_{L_1(\mathds{R}^d)} \Big(\|u\|_{L_{p_-}(\Om)}^{p_-}+\|u\|_{L_{p_+}(\Om)}^{p_+}\Big).
\end{align*}
Hence, $L_{p_+}(\Om)\cap L_{p_-}(\Om)$ is  a subspace of $\cL(\Om)$  and we arrive at the left embedding in (\ref{2.4}). The above estimate and the left inequality in (\ref{3.17a}) for $u$ with $\|u\|_{\cL(\Om)}<1$ yield
\begin{equation*}
\|u\|_{\cL(\Om)} \leqslant 2 
\big(\b c_1  \|a\|_{L_1(\mathds{R}^d)}\big)^{\frac{1}{p_+}} \Big(\|u\|_{L_{p_-}(\Om)}^{p_-}+\|u\|_{L_{p_+}(\Om)}^{p_+}\Big)^{\frac{1}{p_+}}+
\|u\|_{L_{p_-}(\Om)}^{p_-},
\end{equation*}
and this means that the left embedding in (\ref{2.4}) is continuous.
The proof is complete.
\end{proof}

\begin{proof}[Proof of Theorem~\ref{th4}]
As it has been said in the beginning of this section, the functionals $F\left(\frac{u}{\l}\right)$ and $G\left(\frac{u}{\l}\right)$  are monotonically decreasing and continuous in $\l,$ the value  of the functional $\l_F:=|u|_{p,\Om},$ is the unique solution of the equation
\begin{equation}\label{3.11}
F\left(\frac{u}{\l_F}\right)=1,
\end{equation}
while  $\l_G:=g(u)$ in (\ref{2.7})
is the unique solution of the equation
\begin{equation}\label{3.12}
G\left(\frac{u}{\l_G}\right)=1.
\end{equation}
It follows from the definition of the functional $G$ in (\ref{2.7}) that
\begin{equation}\label{3.14}
F\left(\frac{u}{\l_G}\right)\leqslant 1=F\left(\frac{u}{\l_F}\right),
\qquad \int\limits_{\Om} \left|\frac{u}{\l_G}\right|^{p_-}\,dx\leqslant 1,
\end{equation}
and by the mentioned monotonicity of the functional $F\left(\frac{u}{\l}\right)$ we immediately conclude that
\begin{equation}\label{3.13}
\l_G\geqslant \l_F,
\end{equation}
while the second inequality in (\ref{3.14}) yields
\begin{equation}\label{3.15}
\l_G\geqslant \|u\|_{L_{p_-}(\Om)}.
\end{equation}
These two inequalities imply the left inequality in (\ref{2.6}).

It remains to prove the right inequality in (\ref{2.6}). Since
$\l_G/\l_F\geqslant 1$
and $1<p_-\leqslant p_+,$ by the obvious identity
\begin{equation*}
F\left(\frac{u}{\l_G}\right)=F\left(\frac{u}{\l_F} \frac{1}{\l_G/\l_F}\right)
\end{equation*}
and the left inequality in (\ref{3.22a}) we get
\begin{equation*}
F\left(\frac{u}{\l_G}\right)\leqslant \b F\left(\frac{u}{\l_F}\right) \frac{\l_F^{p_-}}{\l_G^{p_-}}  = \b \frac{\l_F^{p_-}}{\l_G^{p_-}}.
\end{equation*}
Hence,
\begin{equation*}
1=G\left(\frac{u}{\l_G}\right)\leqslant \b\frac{\l_F^{p_-}}{\l_G^{p_-}} + \frac{1}{\l_G^{p_-}} \|u\|_{L_{p_-}(\Om)}^{p_-},
\end{equation*}
which yields
\begin{equation*}
\l_G^{p_-} \leqslant \b\l_F^{p_-}+\|u\|_{L_{p_-}(\Om)}^{p_-} \leqslant \big(\b^{\frac{1}{p_-}}\l_F+\|u\|_{L_{p_-}(\Om)}\big)^{p_-}.
\end{equation*}
This estimate implies the right inequality in (\ref{2.6}). The proof is complete.
\end{proof}


\section{Density of smooth functions and separability}

In this section we prove Theorem~\ref{th2}. We first observe that since the set of
compactly supported essentially bounded functions is a subset of $L_{p_+}(\Om)\cap L_{p_-}(\Om),$ by (\ref{2.4}) this set is also a subset of
$\cL(\Om).$
This implies that the space of infinitely differentiable compactly supported functions is a subset  of $\cL(\Om).$ It follows from (\ref{3.30}) that
\begin{equation*}
\vp\big(|u(x)-u(y)|,x,y\big)\leqslant  2^{p_+-1}\b  \Big(\vp\big(|\RE u(x)-\RE u(y)|,x,y\big)+\vp\big(|\IM u(x)-\IM u(y)|,x,y\big) \Big).
\end{equation*}
%
Hence, it is sufficient to prove the theorem only for real--valued functions $u.$ First we are going to  show that each real--valued  function $u\in \cL(\Om)$ can be arbitrarily close approximated by a compactly supported continuous function. We choose an arbitrary large natural $n,$   denote
\begin{equation*}
\Om^n:=\{x\in\Om:\, |u(x)|<n\},\qquad \Om^n_\pm:=\{x\in\Om:\, \pm u(x)>n\},
\end{equation*}
and introduce a sequence of functions
\begin{equation*}
u_n(x):=\left\{
\begin{aligned}
& u(x),\qquad x\in\Om^n,
\\
& \pm n,\qquad x\in\Om^n_\pm.
\end{aligned}\right.
\end{equation*}
It is straightforward to verify that
\begin{align*}
\big|u(x)-u_n(x)-(u(y)-u_n(y))\big|=& \left\{
\begin{aligned}
&0, &&  (x,y)\in\Om^n\times\Om^n,
\\
|u(y)&\mp n|, && (x,y)\in\Om^n\times\Om^n_\pm,
\\
|u(x)&\mp n|, && (x,y)\in\Om^n_\pm\times\Om^n,
\\
|u(x)&-u(y)|, && (x,y)\in\Om^n_\pm\times\Om^n_\pm,
\\
\big|u(x)-&u(y)\mp2n\big|, \quad&& (x,y)\in\Om^n_\pm\times\Om^n_\mp,
\end{aligned}\right.
\\
\leqslant & \left\{
\begin{aligned}
&0, &&  (x,y)\in\Om^n\times\Om^n,
\\
|u(x)&-u(y)|,\quad && (x,y)\in\Om\times\Om\setminus(\Om^n\times\Om^n).
\end{aligned}\right.
\end{align*}
This inequality and the monotonicity of the function $\vp(t,x,y)$ in $t,$ see Lemma~\ref{lm:MonCon}, imply
\begin{equation*}
F(u-u_n)\leqslant \int\limits_{\Om\times\Om\setminus(\Om^n\times\Om^n)}
\vp\big(|u(x)-u(y)|,x,y\big)a(x-y)\,dxdy.
\end{equation*}
The sequence of domains $\Om^n\times\Om^n$ monotonically exhausts $\Om\times\Om$ and the above inequality implies that $F(u-u_n)\to+0$ as $n\to+\infty.$ Hence, the functions $u_n$ belong to $\cL(\Om)$ and $u_n\to u$ in $\cL(\Om)$ as $n\to+\infty.$ Thus, each function $u\in \cL(\Om)$ can be approximated by a bounded function and in what follows we can suppose that $u\in \cL(\Om)$ is bounded.

Now we choose an arbitrary positive $\e\in(0,1)$ and denote
\begin{equation*}
\Om_\e:=\big\{x\in\Om:\,  0<|u(x)|<\e\big\}.
\end{equation*}
Since $u\in L_{p_-}(\Om),$ the domain $\Om\setminus\Om_\e$  has a finite measure. If the domain $\Om$ has an infinite measure, then the same is true for $\Om_\e.$
We define the function
\begin{equation*}
u_\e(x):=\left\{
\begin{aligned}
u&(x),\qquad x\in\Om\setminus\Om_\e,
\\
&0,\qquad \hphantom{x,}x\in\Om_\e.
\end{aligned}\right.
\end{equation*}
It is straightforward to confirm that
\begin{equation*}
u(x)-u_\e(x)-(u(y)-u_\e(y))= \left\{
\begin{aligned}
& 0,  && (x,y)\in (\Om\setminus\Om_\e) \times (\Om\setminus\Om_\e),
\\
u&(x), && (x,y)\in \Om_\e \times (\Om\setminus\Om_\e),
\\
-u&(y), && (x,y)\in (\Om\setminus\Om_\e) \times \Om_\e,
\\
u(x)&-u(y),\quad && (x,y)\in \Om_\e  \times \Om_\e.
\end{aligned}
\right.
\end{equation*}
Since
\begin{equation*}
|u(x)|\leqslant \e<1\quad\text{on}\quad  \Om_\e \times (\Om\setminus\Om_\e),\qquad
|u(y)|\leqslant \e<1\quad\text{on}\quad(\Om\setminus\Om_\e)\times  \Om_\e,
\end{equation*}
by Lemma~\ref{lm:MonCon}, the inequality (\ref{2.28}) and the right inequality in (\ref{3.28})
we obtain
\begin{equation*}
\vp\big(|u(x)-u_\e(x)-(u(y)-u_\e(y))\big|,x,y\big)\leqslant \left\{
\begin{aligned}
& 0,  && (x,y)\in (\Om\setminus\Om_\e) \times (\Om\setminus\Om_\e),
\\
\b c_1&|u(x)|^{p_-}, && (x,y)\in \Om_\e \times (\Om\setminus\Om_\e),
\\
\b c_1&|u(y)|^{p_-}, && (x,y)\in (\Om\setminus\Om_\e) \times \Om_\e,
\\
\vp\big(|u(x)&-u(y)|,x,y\big),\quad && (x,y)\in \Om_\e \times \Om_\e.
\end{aligned}
\right.
\end{equation*}
We multiply this inequality by $a(x-y)$ and integrate it  over $\Om\times\Om.$ This gives:
\begin{align*}
F(u-u_\e)\leqslant & \b c_1\int\limits_{ \Om_\e \times (\Om\setminus\Om_\e) } |u(x)|^{p_-}
a(x-y) \,dxdy +\b c_1\int\limits_{(\Om\setminus\Om_\e) \times \Om_\e } |u(y)|^{p_-}
a(x-y) \,dxdy
\\
&+\int\limits_{ \Om_\e  \times \Om_\e } \vp\big(|u(x)-u(y)|,x,y\big)
a(x-y)\,dxdy
\\
\leqslant & 2 \b c_1 \|a\|_{L_1(\mathds{R}^d)}
\|u\|_{L_{p_-}(\Om_\e)}^{p_-} +  \int\limits_{ \Om_\e  \times \Om_\e } \vp\big(|u(x)-u(y)|,x,y\big) a(x-y)\,dxdy.
\end{align*}
As $\e\to+0$, the domains $(\Om\setminus\Om_\e)\times(\Om\setminus\Om_\e)$ and $\Om\setminus\Om_\e$ monotonically exhaust $\Om\times\Om$ and $\Om.$  Hence, the integrals in the right hand side in the above inequality tend to zero as $\e\to+0.$ This implies that $u-u_\e$ and, hence, $u_\e$ belong to $\cL(\Om)$ for each $\e\in(0,1)$ and $\|u-u_\e\|_{\cL(\Om)}\to0$ as $\e\to+0.$  Thus, we can approximate each function $u\in\cL(\Om)$ by a bounded function, the support of which has a finite measure, and in what follows we suppose that $u$ already possesses such properties.

We recall that the essential support for a function $u\in L_{1,loc}(\Om)$ is defined as the smallest closed subset, denoted by  $\esssupp u,$ such that $u$ vanishes almost everywhere in $\Om$ outside $\esssupp u.$
We denote $S:=\esssupp u,$ $\mes S<+\infty,$ $u\in L_\infty(\Om)$ and suppose that $S$ is unbounded. Then for each $\e>0$ there exists a bounded measurable subdomain  $S_\e\subset S$ such that $\mes S\setminus S_\e<\e$ and $S_{\e_1}\subset S_{\e_2}$ if $\e_1<\e_2.$ We define a function
\begin{equation*}
u^\e(x):=\left\{
\begin{aligned}
u&(x),\ \; && x\in S_\e,
\\
& 0, && x\in \Om\setminus S_\e.
\end{aligned}\right.
\end{equation*}
We have
\begin{equation*}
u(x)-u^\e(x)-(u(y)-u^\e(y))= \left\{
\begin{aligned}
-u&(y),&& (x,y)\in S_\e\times (S\setminus S_\e),
\\
u&(x),&& (x,y)\in (S\setminus S_\e)\times  S_\e,
\\
|u(x)&-u(y)|,&& (x,y)\in (S\setminus S_\e)\times  (S\setminus S_\e),
\\
&0,&& \text{otherwise}.
\end{aligned}\right.
\end{equation*}
This identity, the right inequality in (\ref{3.28}), Lemma~\ref{lm:MonCon} and boundedness of $u$ imply:
\begin{align*}
F(u-u^\e)=& \int\limits_{S_\e\times (S\setminus S_\e)} \vp\big(|u(y)|,x,y\big) a(x-y) \,dxdy
\\
&+  \int\limits_{  (S\setminus S_\e)\times S_\e} \vp\big(|u(x)|,x,y\big) a(x-y)\,dxdy
\\
&+  \int\limits_{ (S\setminus S_\e)\times  (S\setminus S_\e)} \vp\big(|u(x)-u(y)|,x,y\big) a(x-y) \,dxdy
\\
\leqslant & 2 c_1 \big(1+\|u\|_{L_\infty(\Om)}^{p_+}\big)
\|a\|_{L_1(\mathds{R}^d)} \mes S\setminus S_\e
\\
&+ \int\limits_{ (S\setminus S_\e)\times  (S\setminus S_\e)} \vp\big(|u(x)-u(y)|,x,y\big) a(x-y) \,dxdy.
\end{align*}
As $\e\to+0,$ the measure $\mes S\setminus S_\e$ tends to zero, while the domains $(S\setminus S_\e)\times  (S\setminus S_\e)$ exhaust $S\times S.$ The above estimates imply that $u^\e$ belongs to $\cL(\Om)$ and $\|u^\e-u\|_{\cL(\Om)}\to0$ as $\e\to+0.$
The function $u^\e$ is bounded and compactly supported and in what follows we can suppose that $u$ possesses the same properties.

Given a function $u\in \cL(\Om),$ which is bounded and compactly supported in $\Om$,
we continue this function by zero outside $\Om$ and consider a smoothing  by means of an appropriate convolution
\begin{equation*}
\tilde{u}^\e(x):=\int\limits_{\mathds{R}^d} \rho_\e(x-y)  u (y)\,dy,\qquad \rho_\e(x):=\e^{-d} \rho(x\e^{-1}),\qquad \int\limits_{\mathds{R}^d} \rho(x)\,dx=1,
\end{equation*}
where $\rho=\rho(x)$ is a non--negative infinitely differentiable function supported in the ball $\{x:\, |x|\leqslant 1\}.$  We denote $S:=\esssupp u\subset\Om,$ and $S$ is a bounded set. The well--known properties of smoothing state that $\tilde{u}^\e\in C_0^\infty(\mathds{R}^d)$ and
\begin{equation*}
\supp \tilde{u}^\e \subseteq \tilde{S}:=\{x\in\mathds{R}^d:\, \dist(x,S)\leqslant \e\}\subset \Om,\qquad
\|u-\tilde{u}^\e\|_{L_{p_\pm}(\tilde{S})}\to 0,\qquad \e\to+0,
\end{equation*}
Using these convergences, the inequality (\ref{3.30}) and the right inequality in (\ref{3.28}),
we obtain
\begin{equation}\label{4.2}
\begin{aligned}
F(u-\tilde{u}^\e)
=&\int\limits_{\tilde{S}\times\tilde{S}}  \vp\big(|u(x)-\tilde{u}^\e(x)-(u(y)-\tilde{u}^\e(y))|,x,y\big)a(x-y)\,dxdy
\\
&+ \int\limits_{\tilde{S}\times(\Om\setminus\tilde{S})} \vp\big( |u(x)-\tilde{u}^\e(x)|,x,y\big)a(x-y)\,dxdy
\\
&+\int\limits_{(\Om\setminus\tilde{S})\times\tilde{S}}
 \vp\big(|u(y)-\tilde{u}^\e(y)|,x,y\big)a(x-y)\,dxdy
 \\
 \leqslant & 2^{p_+-1}\b \int\limits_{\tilde{S}\times\tilde{S}}  \vp\big(|u(x)-\tilde{u}^\e(x)|,x,y\big)a(x-y)\,dxdy
 \\
& +2^{p_+-1}\b
 \int\limits_{\tilde{S}\times\tilde{S}}  \vp\big(|u(y)-\tilde{u}^\e(y)|,x,y\big)a(x-y)\,dxdy
  \\
&+ \int\limits_{\tilde{S}\times(\Om\setminus\tilde{S})} \vp\big( |u(x)-\tilde{u}^\e(x)|,x,y\big)a(x-y)\,dxdy
\\
&+\int\limits_{(\Om\setminus\tilde{S})\times\tilde{S}}
 \vp\big(|u(y)-\tilde{u}^\e(y)|,x,y\big)a(x-y)\,dxdy
 \\
  \leqslant & (2^{p_+}\b+ 2)\b c_1
   \|a\|_{L_1(\mathds{R}^d)}
\Big(\|u-\tilde{u}^\e\|_{L_{p_-}(\tilde{S})}^{p_-}+ \|u-\tilde{u}^\e\|_{L_{p_+}(\tilde{S})}^{p_+}\Big)\to+0
\end{aligned}
\end{equation}
as $\e\to+0.$ Hence, we can approximate the function $u$ by an infinitely differentiable compactly supported function $\tilde{u}^\e$ and this proves the density of $C_0^\infty(\Om)$ in $\cL(\Om).$

We proceed to proving the separability of the space $\cL(\Om).$ Let $\tilde{\Om}_n,$ $n\in\mathds{N},$ be a monotonically increasing sequence of bounded domains covering $\Om,$ that is,
\begin{equation*}
\tilde{\Om}_n\subset\tilde{\Om}_{n+1},\qquad \bigcup\limits_{n\in\mathds{N}}\tilde{\Om}_n=\Om.
\end{equation*}
By $\cP_n$ we denote the set of all polynomials on $\Om_n$ with rational coefficients and each such polynomial is extended by zero outside $\Om_n.$ Since $\cP_n\subset L_\infty(\Om_n)$ and the extensions of all functions in $\cP_n$ are compactly supported in $\Om,$ we have
\begin{equation*}
\cP_n\subset\cL(\Om),\qquad \cP:=\bigcup\limits_{n\in\mathds{N}} \cP_n\subset \cL(\Om).
\end{equation*}
The set $\cP$ is obviously countable.

For each $u\in C_0^\infty(\Om)$ there exists $n$ such that $\supp u \subset \tilde{\Om}_n$ and by the Stone--Weierstrass theorem for each $\e>0$ there exists a function $\hat{u}_\e\in \cP_n$ such that $\|u-\hat{u}\|_{C(\overline{\tilde{\Om}_n})}<\e.$ Proceeding then as in (\ref{4.2}), we see that $F(u-\hat{u}_\e)\to+0$ as $\e\to+0.$ This proves the separability of $\cL(\Om)$ and completes the proof of Theorem~\ref{th2}.

\section{Space $\L(\Om)$}

In this section we study the space $\L(\Om).$ We begin with proving that the functional $|\cdot|_{p,\Om}$ is indeed a norm on $\L(\Om).$

\begin{lemma}\label{lm5.2}
The functional $|\cdot|_{p,\Om}$ is a norm on $\L(\Om).$
\end{lemma}

\begin{proof}
The functional $|\cdot|_{p,\Om}$ is finite on each element of  the space $\cL(\Om).$  Assume that $|\mathbf{U}|_{p,\Om}=0,$ then
$F(u)=0$ for each $u\in\mathbf{U}$ and hence,
\begin{equation*}
\vp\big(|u(x)-u(y)|,x,y\big)a(x-y)=0\quad\text{a.e. in}\quad \Om\times\Om.
\end{equation*}
The inequality (\ref{2.17}) and the left inequality in (\ref{3.28})
 then imply
 \begin{equation}\label{3.18}
u(x)-u(y)=0\quad\text{for a.e.}\quad x\in\Om,\quad y\in\Om,\quad x-y\in B_0.
\end{equation}
Let $\Om_i$ be a connected component of the domain $\Om$ and  $\om\subset\Om_i$ be an arbitrary subset of $\Om_i$ of positive measure such that $\om$ can be put into a ball of  diameter not exceeding the radius of $B_0.$ Then by (\ref{3.18}) we have
\begin{equation*}
u(x)-u(y)=0\quad\text{for a.e.}\quad x,y\in \om.
\end{equation*}
Integrating this identity over $y\in \om,$ we find
\begin{equation*}
u(x)=\frac{1}{\mes\om} \int\limits_{\om} u(y)\,dy\quad \text{for a.e.}\quad x\in \om.
\end{equation*}
Hence, the function $u$ is constant almost everywhere on $\om.$  Then we cover the connected component $\Om_i$ by open domains $\om_j$ and suppose that each of these domains can be put   into a ball of  diameter not exceeding the radius of $B_0.$ The function $u$ is equal to a  constant $C_j$ almost everywhere on each domain $\om_j.$ Since for each domain $\om_j$ there is at least one another domain $\om_k$ such that the intersection $\om_j\cap\om_k$ has a positive measure, we conclude that $C_j=C_k$ and therefore, the function $u$ is equal to a single constant almost everywhere in $\Om_i.$

Therefore, the function $u$ is constant almost everywhere on each connected component of $\Om,$ but the value of this constant can depend on the component. Let $\Om_i$ and $\Om_{i+1}$ be two connected components of $\Om_i$ and $u=\tilde{C}_j,$ $\tilde{C}_j=const,$ almost everywhere on $\Om_j,$ $j=i,\,i+1.$ Owing to the condition~(\ref{2.19}),  we can choose two subsets $\om_j\subset\Om_j,$ $j=i,\,i+1,$ of positive measures such that $\dist(\om_i,\om_{i+1})<\diam B_0.$ Taking then $x\in\om_i,$ $y\in\om_{i+1}$ in (\ref{3.18}), we immediately conclude that $\tilde{C}_i=\tilde{C}_{i+1}.$ Hence, the function $u$ equals to a single  constant almost everywhere on   $\Om$ and therefore, $\mathbf{U}=0.$ The homogeneity for the functional $|\cdot|_{p,\Om}$  on $\cL(\Om)$
 and triangle inequality,  which were established in the proof of Theorem~\ref{th1}, imply the same properties for the functional $|\cdot|_{p,\Om}$  on $\L(\Om).$  The proof is complete.
\end{proof}

\subsection{Bounded domain with Lipschitz boundary}

In this subsection we prove Theorem~\ref{th5}.
The fact that the functional $|\cdot|_{p,\Om}$ is indeed a norm on $\L(\Om)$ can be proved in the same way as this has been done above for $\cL(\Om).$ It is also clear that
\begin{equation}\label{6.17}
|u|_{p,\Om}\leqslant \|u\|_{\cL(\Om)}\quad\text{for all}\quad u\in \cL(\Om).
\end{equation}

We choose   arbitrary $u\in\L(\Om),$
  $\a\in(0,1)$ and we denote
\begin{equation*}
\Xi_\a:=\left\{(x,y)\in\Om\times\Om:\, \frac{|u(x)-u(y)|}{|u|_{p,\Om}} \geqslant \a
\right\}.
\end{equation*}
Since the measure of $\Om$ is finite,   we can choose $\a\in(0,1)$ small enough but fixed so that
\begin{equation}\label{6.5}
\a^{p_-} \int\limits_{\Om\times\Om}
a(x-y)  \,dxdy\leqslant 1.
\end{equation}
In what follows $\a$ is supposed to satisfy the above estimate.
By the left inequality in (\ref{3.28}) for $(x,y)\in \Xi_\a$ we get
\begin{equation}\label{6.3}
\vp\left(\left|\frac{u(x)-u(y)}{\a|u|_{p,\Om}}\right|,x,y\right)\geqslant  \b^{-1}c_1^{-1} \left|\frac{u(x)-u(y)}{\a|u|_{p,\Om}}\right|^{p_-}.
\end{equation}
We then use the left inequality in (\ref{3.22}) with $\l=\a$ to obtain
\begin{equation*}
\vp\left(\left|\frac{u(x)-u(y)}{ |u|_{p,\Om}}\right|,x,y\right)\geqslant \b^{-2}c_1^{-1} \a^{p_+ - p_-} \left|\frac{u(x)-u(y)}{|u|_{p,\Om}}\right|^{p_-}.
\end{equation*}
Hence,
\begin{equation}\label{6.4}
\int\limits_{\Xi_\a} \vp\left(\left|\frac{u(x)-u(y)}{\a|u|_{p,\Om}}\right|,x,y\right)
 a(x-y)  \,dxdy
\geqslant \b^{-2}c_1^{-1}\a^{p_+ - p_-} \int\limits_{\Xi_\a}
\left|\frac{u(x)-u(y)}{|u|_{p,\Om}}\right|^{p_-}a(x-y) \,dxdy.
\end{equation}
The definition of $\Xi_\a$ implies
\begin{equation}\label{6.1}
\int\limits_{(\Om\times\Om)\setminus\Xi_\a}
\left|\frac{u(x)-u(y)}{|u|_{p,\Om}}\right|^{p_-}a(x-y) \,dxdy \leqslant   \a^{p_-} \int\limits_{(\Om\times\Om)\setminus\Xi_\a}
a(x-y) \,dxdy
\leqslant   1=F\left(\frac{u}{|u|_{p,\Om}}\right).
\end{equation}

We denote
\begin{equation*}
F_{p_-}(u):= \int\limits_{\Om\times\Om}
 |u(x)-u(y)|^{p_-} a(x-y) \,dxdy
\end{equation*}
The estimates (\ref{6.1}) and (\ref{6.4}) yield
\begin{equation*}
F_{p_-}\left(\frac{u}{|u|_{p,\Om}}\right)\leqslant \big(1+\b^2c_1\a^{p_- - p_+}\big) F\left(\frac{u}{|u|_{p,\Om}}\right)=1+\b^2c_1\a^{p_- - p_+},
\end{equation*}
and hence
\begin{equation}\label{6.6}
F_{p_-}(u)\leqslant \big(1+\b^2c_1\a^{p_- - p_+}\big)|u|_{p,\Om}^{p_-}.
\end{equation}
Now we apply Proposition~4.2 from \cite[Ch. 4, Sect. 4.4]{AABPT}, which states the existence of a constant $C>0$ independent of $u$ but depending on $\Om$ such that
\begin{equation}\label{6.7}
\|u_{\bot,\Om}\|_{L_{p_-}(\Om)}^{p_-}\leqslant C F_{p_-}(u),\qquad u_{\bot,\Om}:=u-\la u\ra_\Om.
\end{equation}
This estimates and (\ref{6.6}) imply that the function $u_{\bot,\Om}$ belongs
to $L_{p_-}(\Om)$ and
\begin{equation}\label{6.8}
  \|u_{\bot,\Om}\|_{L_{p_-}(\Om)}\leqslant C |u|_{p,\Om}
\end{equation}
with some $C$ independent of $u.$
By the H\"older inequality we also have
\begin{equation*}
|\la u\ra_\Om|\leqslant (\mes \Om)^{1-\frac{1}{p_-}}\|u\|_{L_{p_-}(\Om)}
\end{equation*}
This inequality, the obvious identity
\begin{equation}\label{6.2}
F(u)=F(u_{\bot,\Om})
\end{equation}
 and  (\ref{6.8}) imply the left inequality in (\ref{2.13}). The right inequality in (\ref{2.13}) follows from (\ref{6.2})  and
\begin{equation*}
\|u\|_{L_{p_-}(\Om)}=\|u_{\bot,\Om}+\la u\ra_\Om\|_{L_{p_-}(\Om)}\leqslant \|u_{\bot,\Om}\|_{L_{p_-}(\Om)}+|\la u\ra_\Om| (\mes\Om)^{\frac{1}{p_-}}.
\end{equation*}
The proof is complete.

\subsection{Arbitrary domain}

In this section we prove Theorem~\ref{th6}. It follows from the definition (\ref{2.3}), (\ref{2.16}) of functional $|\cdot|_{p,\Om}$ on $L_{1,loc}(\Om)$ and $\L(\Om)$ and the definitions (\ref{3.8}), (\ref{3.9}) of  functional $h$ and space $\cH(\Om\times\Om)$ that for each $\mathbf{U}\in\L(\Om)$ we have $|\mathbf{U}|_{p,\Om}=h(U),$ where
 \begin{equation}\label{6.9}
 U=U(x,y):=u(x)-u(y)
\end{equation}
for each $u\in\mathbf{U}$ and the above definition of function $U$ is independent of the particular choice of an element $u$ in the coset $\mathbf{U}\in \L(\Om).$ Hence,  the space $\L(\Om)$ can be identified with the functions $U\in \cH(\Om\times\Om),$ which obey the representation (\ref{6.9}) with some function $u\in\mathbf{U},$ $\mathbf{U}\in\L(\Om).$  Reproducing literally the proof of Lemma~3.6.6 in   \cite[Ch. I\!I\!I, Sect. 3.6]{HH}, we see that  the norm  $|u|_{p,\Om}$ is  uniformly convex on $\L(\Om).$

We proceed to proving  that   the space $\L(\Om)$ is Banach.
Since $\Om$ is an open domain, its connected component $\Om_1,$ see (\ref{2.19}),
  contains a non--empty ball $B$ of a diameter  not exceeding $\frac{1}{2}\diam B_0.$

\begin{lemma}\label{lm5.4} Let $\om\subset\Om$ be a bounded domain with an infinitely differentiable boundary such that $B\subset \om.$ Then each $\mathbf{U}\in\L(\Om)$ contains a unique representative $u$ such that
\begin{equation}\label{6.18}
\la u\ra_B=0
\end{equation}
and the estimate
\begin{equation}\label{6.19}
 |\la u\ra_{\om}|\leqslant c_{11}\|U\|_{\L(\om)}
\end{equation}
holds, where $c_{11}$ is some constant independent of $u,$ but depending on $\om$ and $B.$
\end{lemma}

\begin{proof}
It is obvious that the condition (\ref{6.18}) determines uniquely an element in each coset $\mathbf{U}\in\L(\Om).$ Since the domain $\om$ is bounded and its domain is infinitely differentiable,    we can cover this domain by finitely many intersecting balls $B_j,$ $j=1,\ldots,N,$ such that
\begin{equation}\label{6.23}
\diam B_j\leqslant \frac{1}{2}\diam B_0,\qquad \mes B_j\cap\om\geqslant c_{12}>0,\qquad j=1,\ldots,N,
\end{equation}
where $c_{12}$ is some constant independent of $j.$

Integrating the difference $u(x)-u(y)$ over $x\in B$ and $y\in\om,$ in view of the condition (\ref{6.18}) we easily find
\begin{equation*}
-\la u\ra_\om\mes \om \mes B= \int\limits_{B} dx\int\limits_{\om} (u(x)-u(y))\,dy
\end{equation*}
and hence,
\begin{equation}\label{6.24}
|\la u\ra_\om|\leqslant \frac{1}{\mes \om\mes B}\int\limits_{B} dx\int\limits_{\om} |u(x)-u(y)|\,dy\leqslant \frac{1}{\mes\om\mes B} \sum\limits_{j=1}^{N} \int\limits_{B} dx\int\limits_{B_j\cap\om} |u(x)-u(y)|\,dy.
\end{equation}
For each $j=1,\ldots,N$ we choose $l_j$ balls $B_{k_i},$ $i=1,\ldots l_j,$ such that the intersections
\begin{equation*}
B_j\cap\om\cap B_{k_1},\qquad B_{k_i}\cap B_{k_{i+1}},\quad i=1,\ldots,l_j-1,\qquad B_{k_{l_j}}\cap B
\end{equation*}
have positive measures and each of them can be put into a ball of a diameter not exceeding $\frac{1}{2}\diam B_0.$ Taking into consideration (\ref{6.23}), (\ref{2.17}),   by straightforward calculations we obtain
\begin{align*}
\int\limits_{B} dx\int\limits_{B_j\cap\om} |u(x)-u(y)|\,dy=&\prod\limits_{i=1}^{l_j}\frac{1}{\mes B_{k_i}\cap\om}
\int\limits_{B_{k_1}}dz_1 \cdots \int\limits_{B_{k_{l_j}}}dz_{l_j}
\int\limits_{B} dx\int\limits_{B_j\cap\om} |u(x)-u(y)|\,dy
\\
\leqslant & \frac{1}{c_0 c_{12}^{l_j}} \int\limits_{B_{k_1}\cap \om}dz_1 \cdots \int\limits_{B_{k_{l_j}}\cap \om}dz_{l_j}
\int\limits_{B} dx\int\limits_{B_j\cap\om} |u(y)-u(z_1)|a(y-z_1)\,dy
\\
&+\sum\limits_{i=1}^{l_j-1} \frac{1}{c_0  c_{12}^{l_j}}\int\limits_{B_{k_1}\cap \om}dz_1 \cdots \int\limits_{B_{k_{l_j}}\cap \om}dz_{l_j}
\int\limits_{B} dx\int\limits_{B_j\cap\om} |u(z_i)-u(z_{i+1})| a(z_i-z_{i+1})\,dy
\\
&+ \frac{1}{c_0   c_{12}^{l_j}} \int\limits_{B_{k_1}\cap \om}dz_1 \cdots \int\limits_{B_{k_{l_j}}\cap \om}dz_{l_j}
\int\limits_{B} dx\int\limits_{B_j\cap\om} |u(z_{l_j})-u(x)|a(z_{l_j}-x)\,dy
\\
\leqslant & C \int\limits_{B_{k_1}\cap \om}  \int\limits_{B_j\cap\om} |u(y)-u(z_1)| a(y-z_1)\, dz_1 dy
\\
&+ C  \int\limits_{B_{k_{l_j}}\cap \om}
\int\limits_{B} |u(z_{l_j})-u(x)|a(z_{l_j}-x) \, dz_{l_j}dx
\\
&+ C\sum\limits_{i=1}^{l_j-1} \int\limits_{B_{k_i}\cap \om}  \int\limits_{B_{k_{i+1}}\cap \om}|u(z_i)-u(z_{i+1})| a(z_i-z_{i+1})\, dz_i dz_{i+1}
\\
\leqslant & C \int\limits_{\om\times\om} |u(x)-u(y)| a(x-y)\,dxdy, 
\end{align*}
where $C$ are some absolute constants independent of $u.$ Using estimates (\ref{3.20a}), (\ref{6.24}), we obtain (\ref{6.19}) and complete the proof.
\end{proof}

We choose a fundamental sequence $\mathbf{U}_n$ in $\L(\Om)$ and fix the elements $u_n\in\mathbf{U}_n$ by the condition (\ref{6.18}). Let $\Om_m\subseteq\Om,$ $m\in\mathds{N},$ be a monotone sequence of bounded subdomains of $\Om$ with infinitely differentiable boundaries, which exhausts the domain $\Om,$ namely,
\begin{equation}\label{6.15}
\om\subset\Om_m,\quad \Om_m\subseteq\Om_{m+1},\quad m\in\mathds{N},\qquad \bigcup\limits_{m\in\mathds{N}} \Om_m=\Om.
\end{equation}
The monotonicity of the functional $F(\frac{u}{\l})$ in $\l$ established in Section~\ref{sec2} and Equation~(\ref{3.16}) for the norm $|\cdot|_{p,\Om}$ imply
\begin{equation}\label{6.20}
|u|_{p,\Om_m}\leqslant |u|_{p,\Om_{m+1}}\leqslant |u|_{p,\Om},\qquad m\in\mathds{N}.
\end{equation}
This inequality yields that
\begin{equation}\label{6.21}
|u_n|_{p,\Om_m}\leqslant |u_n|_{p,\Om}\leqslant C,\qquad |u_n-u_k|_{p,\Om_m}\leqslant |u_n-u_k|_{p,\Om}\to0,\quad n,k\to+\infty,
\end{equation}
uniformly  in $m\in\mathds{N},$ where $C$ is some constant independent of $n$ and $m.$  Then it follows from (\ref{6.19}) and (\ref{2.13}) with $\Om$ replaced by $\Om_m$ that for each $m\in\mathds{N}$  the sequence $u_n$ is fundamental in $\cL(\Om_m).$ Hence, by Theorem~\ref{th5}, it converges in $\cL(\Om_m),$ $L_{p_-}(\Om_m)$ and almost everywhere in $\Om_m$ to some function $u^{(m)}.$ It is clear that $u^{(m+1)}=u^{(m)}$ on $\Om_m$ and this is why the function $u(x):=u^{(m)}(x),$ $x\in\Om_m,$ is well--defined on $\Om,$ and it is a limit of the sequence $u_n$ in the sense of almost everywhere convergence in $\Om$ and of the convergence in $L_{p_-}(\Om_m)$ and $\L(\Om_m)$ for each $m.$ Therefore, for each $m\in\mathds{N}$ there exists $N_1=N_1(m)$ such that
\begin{equation}\label{6.22}
|u_{N_1}-u|_{p,\Om_m}\leqslant 1.
\end{equation}
By (\ref{6.21})   this inequality yields
\begin{align*}
&|u|_{p,\Om_m}\leqslant |u-u_{N_1}|_{p,\Om_m}+|u_{N_1}|_{p,\Om_m} \leqslant C+1,
\\
&\int\limits_{\Om_m\times\Om_m}  \vp\left(\left|\frac{u(x)}{C+1}-\frac{u(y)}{C+1}\right|,x,y\right)a(x-y)\,dxdy\leqslant 1,
\end{align*}
uniformly in $m\in\mathds{N}.$  Passing to the limit as $m\to+\infty$ in the second inequality and using (\ref{6.15}), we find that
\begin{equation*}
F\left(\frac{u}{C+1}\right)\leqslant 1
\end{equation*}
and hence, $u\in\L(\Om).$

The convergence in (\ref{6.21}) implies that for each $\d>0$ there exists $N=N_2(\d)$ such that for $n,k\geqslant N_2$ we have
\begin{equation*}
|u_n-u_k|_{p,\Om_m}\leqslant \d
\end{equation*}
uniformly in $m\in\mathds{N}.$ We pass to limit as $k\to+\infty$ in this inequality and we get
\begin{equation*}
|u_n-u|_{p,\Om_m}\leqslant \d
\end{equation*}
for $n\geqslant N_2(\d)$ uniformly in $m\in\mathds{N}.$ Hence,
\begin{equation*}
  \int\limits_{\Om_m\times\Om_m} \vp\left( \left|\frac{u_n(x)-u(x)}{\d}-\frac{u_n(y)-u(y)}{\d}\right|,x,y\right) a(x-y) \,dxdy\leqslant 1
\end{equation*}
uniformly in $m\in\mathds{N}.$ Passing to the limit as $m\to+\infty$ and using (\ref{6.15}), we obtain
\begin{equation*}
F\left(\frac{u_n-u}{\d}\right)\leqslant 1
\end{equation*}
and in view of the monotonicity of the functional $F\left(\frac{u}{\l}\right)$ in $\l$ and Equation~(\ref{3.16}) with $u$ replaced by $u_n-u$ we have the estimate
\begin{equation*}
|u_n-u|_{p,\Om}\leqslant \d
\end{equation*}
for $n\geqslant N_2(\d).$ This means that $|u_n-u|_{p,\Om}\to+0$ as $n\to+\infty.$  The proof is complete.

\section{Dual spaces}\label{sec7}

In this section we  describe the dual spaces of $\L(\Om)$ and $\cL(\Om).$

\subsection{Auxiliary statements}

A key ingredient in the description of the dual space 
is the following lemma, which is of an independent interest.

\begin{lemma}\label{lm5.1}
Let $X$ be a Banach space with a uniformly convex norm.
Then for each functional $\phi\in X^*$  there exists a unique $w\in X$ such that
\begin{equation}\label{5.3}
\|w\|_X=\|\phi\|_{X^*},\qquad\phi(w)=\|\phi\|_{X^*}^2.
\end{equation}
Suppose, in  addition,  that for all $w,\,v\in X$ obeying the inequalities
\begin{equation}\label{5.21}
w\ne0,\qquad  \|w + t v\|_X\geqslant \|w\|_X,\quad t\in[0,t_0],
\end{equation}
the estimate
\begin{equation}\label{5.1}
\|w + t v\|_X\leqslant \|w\|_X + t \ell(w,v)+\g(t;w,v),
\end{equation}
is satisfied for $t\in[0,t_0],$ where $t_0=t_0(w,v),$ and $\ell(w,v)$ is some functional on $X\times X$ obeying the property
 \begin{equation}\label{5.22}
\ell(w,-v)=-\ell(w,v),
\end{equation}
while $\g(t;w,v)$ is some function on $t$ depending on the choice of $w$ and $v$ such that
\begin{equation}\label{5.2}
\lim\limits_{t\to+0} \frac{\g(t;w,v)}{t}=0.
\end{equation}
Then   different functionals in $X^*$ can not have the same function $w$ obeying (\ref{5.3}).
\end{lemma}

\begin{proof}
We adapt the proof of Lemma in \cite[Ch. I\!I, Sect. 2.32]{Ad}. Let $\phi\in X^*$ be a non--zero functional.  There exists a sequence $w_n\in X$ such that
\begin{equation*}
\|w_n\|_X=\|\phi\|_{X^*},\qquad |\phi(w_n)|\to \|\phi\|_{X^*}^2,\qquad n\to+\infty.
\end{equation*}
Without loss of generality we can suppose that $|\phi(w_n)|>\frac{1}{2}\|\phi\|_{X^*}^2$ and replacing  then $w_n$ by $\frac{\overline{\phi(w_n})}{|\phi(w_n)|}w_n,$ we get
\begin{equation}\label{5.4}
\|w_n\|_X=\|\phi\|_{X^*},\qquad  \phi(w_n)\to \|\phi\|_{X^*}^2,\qquad n\to+\infty.
\end{equation}

We are going to show that the sequence $w_n$ is fundamental. Suppose this is not true, then, up to  choosing an appropriate subsequence of $w_n,$ there exists $\e>0$ and
$N>0$ such that
\begin{equation*}
\|w_n-w_m\|\geqslant \e\|\phi\|_{X^*}\quad\text{for}\quad n,m>N.
\end{equation*}
By the uniform convexity of the norm there exists $\d>0$ such that
\begin{equation*}
\Big\|\frac{1}{2}(w_n+w_m)\Big\|_{X}\leqslant (1-\d)\|\phi\|_{X^*}.
\end{equation*}
Hence,
\begin{equation}\label{5.5}
\begin{aligned}
\|\phi\|_{X^*}\geqslant& \phi\left(\frac{w_n+w_m}{\|w_n+w_m\|_{X}}\right)
=\frac{1}{\Big\|\frac{1}{2}(w_n+w_m)\Big\|_{X}} \phi\Big(\frac{1}{2}(w_n+w_m)\Big)
\\
\geqslant & \frac{\phi(w_n)+\phi(w_m)}{2(1-\d)\|\phi\|_{X^*}}\to \frac{\|\phi\|_{X^*}}{1-\d},\qquad n,m\to+\infty.
\end{aligned}
\end{equation}
The obtained contradiction proves that the sequence $w_n$ is fundamental, and
it hence converges to some element $w\in X.$ It follows from (\ref{5.4}) that the element $w$ satisfies (\ref{5.3}).

The  element $w$ is unique. Indeed, if there are two functions $w$ and $\tilde{w}$ obeying (\ref{5.3}), then we apply (\ref{5.5}) with $w_m$ and $w_n$ replaced by $w$ and $\tilde{w}$ and immediately get a contraction.

Let us show that each function $w$ can not be associated  with two distinct functionals under additional assumptions (\ref{5.21}), (\ref{5.1}), (\ref{5.22}), (\ref{5.2}).  We again  argue by contradiction supposing that the relations (\ref{5.3}) hold with $\phi=\phi_1$ and $\phi=\phi_2,$ where $\phi_1$ and $\phi_2$ are two distinct functionals. Due to the first identity in (\ref{5.3}) we necessarily have
\begin{equation*}
0<\|w\|_X=\|\phi_1\|_{X^*}=\|\phi_2\|_{X^*}=:c.
\end{equation*}
Since $\phi_1\ne \phi_2,$   there exists $u\in X$ such that $\phi_1(u)\ne \phi_2(u).$ Replacing $u$ by an appropriate linear combination of $u$ and $w,$ we can always achieve the identities
\begin{equation*}
\phi_1(u)=c,\qquad \phi_2(u)=-c.
\end{equation*}
Hence,
\begin{equation}\label{5.6}
\phi_1(w+tu)=\phi_2(w-t u)=c(1+t).
\end{equation}
The definition of the norm of functional implies
\begin{equation*}
c(1+t)\leqslant \|w+tu\|_{X},\qquad c(1+t)\leqslant \|w-tu\|_{X},
\end{equation*}
where $t\in[0,t_0].$ Hence, conditions (\ref{5.21}) are satisfied with $v=u$ and $v=-u$ and it follows from (\ref{5.1}), (\ref{5.22}) that
\begin{equation*}
\|w + t u\|_X\leqslant \|w\|_X + t \ell(w,u)+\g(t;w,u), \qquad  \|w - t u\|_X\leqslant \|w\|_X - t \ell(w,u)+\g(t;w,-u).
\end{equation*}
Summing these two inequalities, we get
\begin{equation*}
2c+2ct\leqslant 2\|w\|_{X}+\g(t,u,v)+\g(-t,u,v)=2c+o(t),
\end{equation*}
and this inequality is impossible. The proof is complete.
\end{proof}

The next step in the proof of Theorem~\ref{th3} is to show that the norm of the space  $\cL(\Om)$ satisfies the assumptions of Lemma~\ref{lm5.1}. In order to do this, we first establish the following lemma.

\begin{lemma}\label{lm5.3}
The following statements hold:
\begin{enumerate}
\item\label{lm5.3-2} For all $U,\,V \in \cH(\Om\times\Om)$ the inequality
\begin{equation}\label{5.15}
\bigg|\int\limits_{\Om\times\Om} \vp'(|U(x,y)|,x,y) |V(x,y)| a(x-y) \,dxdy\bigg|\leqslant C \max\big\{ h^{p_+-1}(U), h^{p_--1}(U)\big\}h(V)
\end{equation}
holds with a constant $C$ independent of $U$  and $V.$

\item\label{lm5.3-3}  For all $\mathbf{U},\,\mathbf{V}\in \L(\Om),$ $\mathbf{U}\ne0,$ and all $u\in\mathbf{U},$ $v\in\mathbf{V}$
 the identity
    \begin{equation}\label{5.12}
    F(u+v)=F(u)+ \ell(\mathbf{U},\mathbf{V})+o(\|\mathbf{V}\|_{\L(\Om)}),\quad \|\mathbf{V}\|_{\L(\Om)}\to0,
        \end{equation}
    holds, 
    where $\ell(\mathbf{U},\mathbf{V})$ is the  functional
     \begin{equation}\label{5.7}
     \ell(\mathbf{U},\mathbf{V})=\int\limits_{\Om\times\Om} \vp'\big( |u(x)-u(y)|,x,y\big)\frac{\RE \big(u(x)-u(y)\big) \big(\overline{v(x)}-\overline{v(y)}\big)
     }{|u(x)-u(y)|}\,dxdy,
     \end{equation}
     on $\L(\Om)\times\L(\Om)$ independent of the choice of $u\in\mathbf{U},$ $v\in\mathbf{V}.$ The functional $\ell$ is additive and linear over real numbers in $\mathbf{V},$ in particular, it satisfies  (\ref{5.22}).
\end{enumerate}
\end{lemma}

\begin{proof}
We first prove an appropriate H\"older inequality following the lines of proof of Lemma~3.2.11 in \cite[Ch. 3, Sect. 3.2]{HH}.
Let $V\in\cH(\Om\times\Om)$ and $W\in \cH_*(\Om\times\Om)$ be arbitrary functions. Applying the Young inequality (\ref{2.27}), we obtain
\begin{equation*}
\int\limits_{\Om\times\Om} \frac{|V(x,y)|}{h(V)} \frac{|W(x,y)|}{h_*(W)} a(x-y)   \,dxdy \leqslant H\left(\frac{V}{h(V)}\right)+H_*\left(\frac{W}{h_*(W)}\right)=2.
\end{equation*}
Multiplying this inequality by $h(V)h_*(W),$ we get
\begin{equation}\label{6.25}
\Bigg|\int\limits_{\Om\times\Om} W(x,y) \overline {V(x,y)}a(x-y)  \,dxdy
\Bigg|\leqslant 2 h(V) h_*(W).
\end{equation}
For each $U\in \cH(\Om\times\Om)$  the inequality~(\ref{3.31}) implies
\begin{equation*}
\vp_*\big(\vp'\big(|U(x,y)|,x,y\big),x,y\big)\leqslant (c_2-1)\vp\big(|U(x,y)|,x,y\big),
\end{equation*}
and hence,
\begin{equation}\label{6.26}
H_*\big(\vp'(|U|,\,\cdot\,)\big)\leqslant (c_2-1)H(U).
\end{equation}
This yields that the function $\vp'(|U|,\,\cdot\,)$ belongs to $\cH_*(\Om\times\Om).$
Using now the inequalities (\ref{3.17b}) and (\ref{3.17c}) to estimate the left and right hand sides of the above inequality, we find
\begin{equation*}
\b_*^{-1} \min\big\{h_*^{q_-}(\vp'(|U|,\,\cdot\,)), h_*^{q_+}(\vp'(|U|,\,\cdot\,))\big\}
\leqslant  \b (c_2-1) \max\big\{h^{p_-}(U), h^{p_+}(U)\big\},
\end{equation*}
where $\b_*$ is the constant from the inequalities (\ref{2.25}) associated with $\vp_*.$
Therefore,
\begin{equation}\label{6.27}
\begin{aligned}
h_*(\vp'(|U|,\,\cdot\,))\leqslant & C \max\Big\{h^{\frac{p_-}{q_-}}(U), h^{\frac{p_+}{q_-}}(U), h^{\frac{p_-}{q_+}}(U),  h^{\frac{p_+}{q_+}}(U)\Big\}
\\
=& C\max\big\{h^{p_--1}(U), h^{p_+-1}(U)\big\},
\end{aligned}
\end{equation}
where $C$ is a constant independent of $U.$
Applying now the inequality (\ref{6.25}) with $W=\vp'(|U|,\,\cdot\,),$ we arrive at (\ref{5.15}).

We proceed to proving (\ref{5.12}). Without loss of generality we suppose that $\|\mathbf{U}\|_{\L(\Om)}=1.$   Given $u=u(x),$ $v=v(x),$ $u\in\mathbf{U},$ $v\in\mathbf{V},$ we denote
\begin{equation*}
U=U(x,y):=u(x)-u(y), \qquad V=V(x,y):=\frac{v(x)-v(y)}{\|V\|_{\L(\Om)}}, \qquad t:=\|V\|_{\L(\Om)}.
\end{equation*}
We suppose that $t\to0$ and hence, $V\to 0$ almost everywhere on $\Om\times\Om.$
By the differentiability of function $\vp$ stated in Lemma~\ref{lm:MonCon}  the convergence
\begin{equation}
\frac{\vp\big(|U(x,y)+tV(x,y)|,x,y\big)-\vp\big(|U(x,y)|,x,y\big)}{t} -
\vp'(|U(x,y)|,x,y) \frac{\RE (\overline{U}(x,y) V(x,y))}{|U(x,y)|}\to 0 \label{5.11}
\end{equation}
holds almost everywhere in $\Om\times\Om$ as $t\to0.$ The same
 differentiability   allows us to apply the Lagrange's mean value theorem, which gives
\begin{align*}
\frac{1}{t}&\Big(\vp\big(|U(x,y) +t V(x,y)|,x,y\big)-\vp\big(|U(x,y)|,x,y\big)\Big)
\\
&=
\vp'\big(|U(x,y) +\tau V(x,y)|,x,y\big) \frac{\RE (U(x,y)+\tau V(x,y)) \overline{V(x,y)}}{|U(x,y)+\tau V(x,y)|}
\end{align*}
for almost all $(x,y)\in\Om\times\Om,$ where $\tau\in(0,t)$ is some intermediate value.
We then estimate the right hand side by means of the Young inequality (\ref{2.27}) and Condition~\ref{Diff}
\begin{align*}
\frac{1}{t}&\Big|\vp\big(|U(x,y) +t V(x,y)|,x,y\big)-\vp\big(|U(x,y)|,x,y\big)\Big|
\\
&\leqslant \vp'\big(|U(x,y) +\tau V(x,y)|,x,y\big) |V(x,y)|
\\
&\leqslant \vp(|V(x,y)|,x,y)+\vp_*\big(\vp'\big(|U(x,y) +\tau V(x,y)|,x,y\big),x,y\big)
\\
&\leqslant \vp(|V(x,y)|,x,y)+(c_2-1)\vp\big(|U(x,y) +\tau V(x,y)|,x,y\big).
\end{align*}
Hence, by (\ref{3.17a}),
\begin{align*}
\int\limits_{\Om\times\Om}& \frac{1}{t}\Big|\vp\big(|U(x,y) +t V(x,y)|,x,y\big)-\vp\big(|U(x,y)|,x,y\big)\Big|a(x-y)\,dxdy
\\
\leqslant  & F(v)+(c_2-1) F(u+\tau v) \leqslant \b \max\{|v|_{p,\Om}^{p_-},\,|v|_{p,\Om}^{p_-}\} + (c_2-1)\b \max\{|u+\tau v|_{p,\Om}^{p_-},\,|u+\tau v|_{p,\Om}^{p_-}\}
\\
\leqslant & C \Big(\max\big\{|u|_{p,\Om}^{p_-},\,|u|_{p,\Om}^{p_-}\big\} +\max\big\{|v|_{p,\Om}^{p_-},\,|v|_{p,\Om}^{p_-}\big\} \Big),
\end{align*}
where $C$ is some constant independent of $u,$ $v$ and $t.$ In the same way we   easily get \begin{equation*}
\int\limits_{\Om\times\Om} \vp'(|U(x,y)|,x,y) \frac{\RE (\overline{U}(x,y) V(x,y))}{|U(x,y)|} \,dxdy
 \leqslant C \Big(\max\big\{|u|_{p,\Om}^{p_-},\,|u|_{p,\Om}^{p_-}\big\} +\max\big\{|v|_{p,\Om}^{p_-},\,|v|_{p,\Om}^{p_-}\big\} \Big).
\end{equation*}
The two above estimates  and convergence (\ref{5.11}) allow us to apply the Lebesgue theorem on the dominated convergence, which gives
\begin{equation*}
\int\limits_{\Om\times\Om} \bigg(\frac{\vp\big(|U(x,y)+t V(x,y)|,x,y\big) -\vp(|U(x,y)|,x,y\big)}{t}  -\vp'\big(|U(x,y)|,x,y\big) \frac{\RE\overline{U(x,y)}V(x,y)}{|U(x,y)|} \bigg)\,dxdy\to 0
\end{equation*}
as $t\to0.$ This convergence implies  (\ref{5.12}), (\ref{5.7}).  The stated properties of the functional $\ell$ are obvious. The proof is complete.
\end{proof}

\subsection{Dual space for $\L(\Om)$}

In this subsection we prove Theorem~\ref{th7}. We begin with verifying the relations (\ref{5.21}), (\ref{5.1}), (\ref{5.2}). Let $\mathbf{U},\,\mathbf{V}\in\L(\Om)$ obey condition (\ref{5.21}) and $u\in\mathbf{U},$ $v\in\mathbf{V}.$ Without loss of generality we can suppose that $\|\mathbf{U}\|_{\L(\Om)}=1,$ and hence, $F(u)=1.$

Let
$\l_t:=|u+t v|_{p,\Om}$ obeys $\l_t\geqslant 1.$ We immediately get the estimate
\begin{equation*}
1\leqslant \l_t \leqslant |u|_{p,\Om}+t|v|_{p,\Om}=1+t|v|_{p,\Om}.
\end{equation*}
Hence,
\begin{equation}\label{5.16}
\frac{1}{\l_t}=1+\eta_t,\qquad \eta_t=O(t),\qquad t\to+0.
\end{equation}
According to (\ref{3.16}), the values $\l_t$ and $\eta_t$ satisfy the identity
\begin{equation}\label{5.17}
1=F\left(\frac{u+tv}{\l_t}\right)=F\big((u+tv)(1+\eta_t)\big).
\end{equation}
We rewrite this identity by means of the identity (\ref{5.12}) and the asymptotic estimate for $\eta_t$ in (\ref{5.16})
\begin{equation*}
1+\ell(\mathbf{U},t\mathbf{V}+\eta_t \mathbf{U})+o(t)=1,
\end{equation*}
and therefore,
\begin{equation*}
 \ell(\mathbf{U},t\mathbf{V})+\frac{\eta_t}{t}\ell(\mathbf{U},\mathbf{U})+o(1)=0,\qquad t\to+0.
\end{equation*}
Passing to the limit as $t\to+0,$ we see that
\begin{equation}\label{5.18}
\frac{\eta_t}{t}=-\frac{ \ell(\mathbf{U},t\mathbf{V})}{\ell(\mathbf{U},\mathbf{U})}+o(1),
\end{equation}
where the denominator
\begin{equation*}
\ell(\mathbf{U},\mathbf{U})=\int\limits_{\Om\times\Om} \vp'\big(|u(x)-u(y)|,x,y\big) |u(x)-u(y)|\,dxdy
\end{equation*}
is positive owing to the positivity of function $t\vp'(t,x,y)$ postulated in Condition~\ref{Diff}.  Substituting (\ref{5.18}) into the identity in (\ref{5.16}) and resolving the resulting relation with respect to  $\l_t,$ we find
\begin{equation}\label{5.19}
\l_t=1+t\frac{ \ell(\mathbf{U},t\mathbf{V})}{\ell(\mathbf{U},\mathbf{U})}+o(t),\qquad t\to+0,
\end{equation}
which yields (\ref{5.1}), (\ref{5.22}), (\ref{5.2}) for the norm $|\,\cdot\,|_{p,\Om}.$

Let $\phi$ be an arbitrary functional in the dual space $(\L(\Om))^*.$ The above established properties of the norm $|\,\cdot\,|_{p,\Om}$ allow us to apply Lemma~\ref{lm5.1} and conclude that there exists a unique element $\mathbf{W}\in \L(\Om)$ obeying (\ref{5.3}) and associated with the functional $\phi$ in the sense of identities (\ref{5.3}).  It is clear that
\begin{equation}\label{5.8}
\|\phi\|=\frac{\phi(\mathbf{W})}{|\mathbf{W}|_{p,\Om}}=
\sup\big\{|\phi(\mathbf{V})|:\, \mathbf{V}\in\L(\Om),\ |\mathbf{V}|_{p,\Om}=1\big\}.
\end{equation}

If for some $\mathbf{V}\in \L(\Om)$ we have $f(v)=1,$ $v\in\mathbf{V},$ then $|\mathbf{V}|_{p,\Om}=|v|_{p,\Om}\leqslant 1.$ At the same time,  owing to the monotonicity of $f\left(\frac{v}{\l}\right)$ in $\l,$
\begin{equation*}
f\left(\frac{v}{\l}\right)>f(v)=1 \quad\text{for} \quad \l<1.
\end{equation*}
Hence, if $f(v)=1,$ then $|v|_{p,\Om}=1.$ And as it follows from  (\ref{3.11}), if $|v|_{p,\Om}=1,$ then $f(v)=1.$ Therefore,
\begin{equation*}
\big\{v\in\cL(\Om):\ |v|_{p,\Om}=1\big\}= \big\{v\in\cL(\Om):\ f(v)=1\big\}
\end{equation*}
and we can rewrite (\ref{5.8}) as
\begin{equation}\label{5.9}
\|\phi\|=\sup\Big\{|\phi(\mathbf{V})|:\, \mathbf{V}\in\L(\Om), \  f(v)=1\ \text{for}\ v\in\mathbf{V} \Big\},
\end{equation}
and the supremum is attained at the unique element $\mathbf{V}=\mathrm{w}_0:=\frac{\mathbf{W}}{|\mathbf{W}|_{p,\Om}}.$

We choose an arbitrary $\mathbf{U}\in \cL(\Om)$ and   consider the functional
$\mathbf{U}\to |\phi(\mathbf{W}_0+\mathbf{U})|.$ Then the point $\mathbf{U}=0$  is the point of the conditional maximum of this function under the restriction is $f(w_0+u)=1,$ $w_0\in\mathbf{W}_0,$ $u\in\mathbf{U}.$ Therefore, by the Lagrange multiplier rule,
\begin{equation*}
\d\big(|\phi(\mathbf{W}_0+\mathbf{U})|+\mu f(w_0+u)\big) =0
\end{equation*}
for some $\mu.$ The variation of the first term reads
\begin{equation*}
\d |\phi(\mathbf{W}_0+\mathbf{U})|=\RE \phi(\d \mathbf{U}),
\end{equation*}
while the derivative of the second term can be easily found by (\ref{5.12}), (\ref{5.7}) and it is equal to $\mu\ell(\mathbf{W}_0,\d \mathbf{U}).$ Hence,
\begin{equation*}
\RE \phi(\d \mathbf{U})=-\mu  \ell(\mathbf{W}_0,\d \mathbf{U}).
\end{equation*}
This is in fact the explicit formula for the action of the functional $\phi,$ that is,
\begin{equation*}
\RE \phi_0(u)=-\mu  \ell(w_0,u),\qquad u\in\cL(\Om).
\end{equation*}
Choosing $\mathbf{U}=\mathbf{W}_0,$ we find $\mu$
 and the final expression for $\RE\phi(\mathbf{U})$
 \begin{equation*}
 \mu=-\frac{\RE\phi(\mathbf{W}_0)}{\ell(\mathbf{W}_0,\mathbf{W}_0)}= - \frac{|\mathbf{W}_0|_{p,\Om}}{\ell(\mathbf{W}_0,\mathbf{W}_0)},\qquad
 \RE\phi(\mathbf{U})=\frac{|\mathbf{W}_0|_{p,\Om}}{\ell(\mathbf{W}_0,\mathbf{W}_0)}
 \ell(\mathbf{W}_0,\mathbf{U}).
 \end{equation*}
Replacing then $\mathbf{U}$ by $-\iu \mathbf{U},$ we get
\begin{align*}
\IM\phi(\mathbf{U})=\frac{|\mathbf{W}_0|_{p,\Om}} {\ell(\mathbf{W}_0,\mathbf{W}_0)}\int\limits_{\Om\times\Om}& \vp'
 \big(|w_0(x)-w_0(y)|,x,y\big)
\\
&\cdot\frac{\IM\Big(\big(\overline{w_0(x)} -\overline{w_0(y)}\big)
     \big(u(x)-u(y)\big)\Big)}{|w_0(x)-w_0(y)|}   a(x-y) \,dxdy,
\end{align*}
where $u\in\mathbf{U},$ $w_0\in\mathbf{W}_0.$
We multiply this identity by $\iu$ and sum with the above formula for $\RE\phi(\mathbf{U}).$ Recalling then the definition (\ref{5.7}) of the functional $\ell,$ we arrive at the formula (\ref{2.20}), (\ref{2.21}).

Given a function $w\in \cL(\Om),$  the function
\begin{equation*}
W(x,y):=\frac{|w|_{p,\Om}}
{\Phi\left(\frac{w}{|w|_{p,\Om}},
\frac{w}{|w|_{p,\Om}}\right)} \vp'\big(|w(x)-w(y)|,x,y\big)
\frac{\overline{w(x)}-\overline{w(y)}}{|w(x)-w(y)|}
\end{equation*}
is an element of the space $\cH_*(\Om\times\Om),$ see (\ref{6.26}), (\ref{6.27}). Hence, the functional $\phi$ defined by (\ref{2.20}) can be represented in the form (\ref{2.22}) with the above function $W.$ At the same time, thanks to the estimate (\ref{5.15}),
 each function $W\in \cH_*(\Om\times\Om)$ generates a linear functional on $\cL(\Om)$ by the formula (\ref{2.20}). If two functions $W_1,\,W_2\in\cH_*(\Om\times\Om)$ generate the same functional $\phi,$ then the difference $W_1-W_2$ generates  the zero functional. It is straightforward to confirm that the formula (\ref{2.20}) generates the zero functional if and only if the function $W$ obeys the condition (\ref{2.14}). This completes the proof.

\subsection{Dual space for $\cL(\Om)$}

In this subsection we study the dual space for $\cL(\Om)$ and prove Theorem~\ref{th8}. The proof is essentially based on the following statement: \textit{if $X$ and $Y$ are two Banach spaces continuously embedded into a Hausdorff topological vector space, then the intersection $X\cap Y$ equipped  with the norm $\|\,\cdot\,\|_X+\|\,\cdot\,\|_Y$ is a Banach space and the dual space for $X\cap Y$ is isomorphic to $(X^*\times Y^*)/Y_\bot,$ where $Y_\bot$ is the subspace of $X^*
\times Y^*$ consisting of the pairs $(\xi,\z)$ such that $\xi(\mathrm{x})+\z(\mathrm{y})=0$ for all $\mathrm{x}\in X,$ $\mathrm{y}\in Y.$}
The proof of this statement can be found, for instance, in \cite[Ch. 1, Sect. 3.2]{KPS}. The space $\cL(\Om)$ can also be treated as the intersection, see Equation~(\ref{2.24}), but only in certain sense. This is why we just briefly recover the proof of the aforementioned statement adapted for our spaces.

We first of all introduce the Cartesian product $\L(\Om)\times L_{p_-}(\Om);$ being equipped  with the norm  $\|\,\cdot\,\|_{\L(\Om)}+\|\,\cdot\,\|_{L_{p_-}(\Om)},$ it is  a Banach space and its dual space reads
\begin{equation*}
\big(\L(\Om)\times L_{p_-}(\Om)\big)^*=\L^*(\Om)\times L_{p_-}^*(\Om).
\end{equation*}
Each functional $(\xi,\z)\in \L^*(\Om)\times L_{p_-}^*(\Om)$ acts on an element $(\mathbf{U},v)\in \L(\Om)\times L_{p_-}(\Om)$ by the rule
\begin{equation*}
(\xi,\z)(\mathbf{U},v)=\xi(\mathbf{U})+\z(v).
\end{equation*}

In this space we introduce the subspace
\begin{equation*}
A:=\big\{(\mathbf{U},u)\in \L(\Om)\times L_{p_-}(\Om):\ u\in\mathbf{U}  \big\}.
\end{equation*}
The subspace $A$ is isometrically isomorphic to $\cL(\Om);$  the isomorphism is defined as
\begin{equation*}
T:\, \cL(\Om)\to A,\qquad Tu:=(\mathbf{U},u),\quad \mathbf{U}:=\{u+C,\ C\in\mathds{C}\}, \qquad T^{-1}(\mathbf{U},u)=u.
\end{equation*}
Then the dual spaces $(\cL(\Om))^*$ and $A^*$ are also isomorphic; the isomorphism reads as
\begin{equation*}
Q:\, (\cL(\Om))^*\to A^*,\qquad Q \xi:=\xi(T^{-1}\,\cdot\,)\qquad Q^{-1}\z=\z(T\,\cdot\,).
\end{equation*}

We then consider the space $A_\bot,$ which is the subspace in $\L^*(\Om)\times L_{p_-}^*(\Om)$ consisting of the functionals vanishing on $A,$ that is,
\begin{equation*}
  A_\bot:=\big\{(\xi,\z):\  \xi(\mathbf{U})+\z(u)=0\quad \text{for all}\quad (\mathbf{U},u)\in A\big\}.
\end{equation*}
The dual space for $A_\bot$ is isometrically isomorphic to the quotient space
\begin{equation*}
(\L^*(\Om)\times L_{p_-}^*(\Om))\, /\,  A_\bot:=\Big\{\big\{(\xi+ \xi_\bot,\z+\z_\bot):\ (\xi_\bot,\z_\bot)\in A_\bot\big\}:\ (\xi,\z)\in \L^*(\Om)\times L_{p_-}^*(\Om) \Big\}.
\end{equation*}
The isomorphism is defined as
\begin{align*}
&Z: \, (\L^*(\Om)\times L_{p_-}^*(\Om))\, /\,  A_\bot \to A^*,
\\
&Z\big\{(\xi+ \xi_\bot,\z+\z_\bot):\ (\xi_\bot,\z_\bot)\in A_\bot\big\}:=(\xi,\z),
\\
&Z^{-1}(\xi,\z):=\{(\xi+ \xi_\bot,\z+\z_\bot):\ (\xi_\bot,\z_\bot)\in A_\bot\big\},
\end{align*}
and this isomorphism is well--defined and is independent of the choice of the pair $(\xi,\z)$ owing to the definition of the subspace $A_\bot.$ Hence, the mapping $Z^{-1} Q$ is an isomorphism between the space $(\cL(\Om))^*$ and $(\L^*(\Om)\times L_{p_-}^*(\Om))\, /\,  A_\bot.$  Given an arbitrary coset
\begin{equation*}
\big\{(\xi+ \xi_\bot,\z+\z_\bot):\ (\xi_\bot,\z_\bot)\in A_\bot\big\}\in (\L^*(\Om)\times L_{p_-}^*(\Om))\, /\,  A_\bot,
\end{equation*}
we can define the corresponding functional on $\phi\in\cL(\Om)$ as
\begin{equation}\label{6.30}
\phi(u)=\xi(\mathbf{U})+\z(u),\qquad u\in \cL(\Om),
\end{equation}
where $\mathbf{U}\in\cL(\Om)$ is the coset generated by $u.$
Hence, each functional in $\cL(\Om)^*$ can be  represented in the form (\ref{6.30}), where $\xi\in(\L(\Om))^*$ and $\z\in (L_{p_-}(\Om))^*.$   The dual space $(\L(\Om))^*$ was described in the previous subsection, while
$(L_{p_-}(\Om))^*=L_{q_-}(\Om).$ Now the representation~(\ref{2.23}) follows from (\ref{6.30}).  The proof is complete.

\subsection{Bounded domain}

In this subsection we prove Theorem~\ref{th3}.
For each  $u\in\cL(\Om)$ we have the representation (\ref{2.12}) and  owing to the inequalities (\ref{2.13}), the space of all elements $u_{\bot,\Om}$ equipped  with the norm $|\cdot|_{p,\Om}$ can be identified with the space $\L(\Om).$ Hence, each bounded linear functional on the space of all elements $u_{\bot,\Om}$ can be represented in the form (\ref{2.20}), (\ref{2.21}).

We choose an arbitrary functional $\phi\in(\cL(\Om))^*$ and apply it to the representation (\ref{2.12})
\begin{equation}\label{5.23}
\phi(u)=\phi(u_\bot) + \la u\ra \phi(\mathds{1}),\qquad \mathds{1}(x)=1,\quad x\in\Om.
\end{equation}
The functional $\phi(u_\bot)$ can be represented in the form (\ref{2.20}), (\ref{2.21}), while the quantity $\phi(\mathds{1})$ is the value of $\phi$ on the constant function $\mathds{1},$   which is identically one on $\Om.$ Then the identity (\ref{5.23}) implies (\ref{2.10a}). The proof is complete.

\section{Class of functions satisfying conditions \ref{Meas}--\ref{Diff}}\label{sec classc1c5}
This section is aimed at proving  Theorems~\ref{l_sum_prod},~\ref{cor_pert} and  Lemma~\ref{lem_sec_deriv}.

\begin{proof}[Proof of Theorem \ref{l_sum_prod}]
It is straightforward to verify that the Conditions \ref{Meas}--\ref{Diff} hold for the sum $\vp(z,x,y)+\psi(z,x,y)$ and $\vp(z,x,y)b(x,y).$

For the product $\vp(z,x,y)\psi(z,x,y)$ and the composition $\vp(\psi(z,x,y),x,y)$ the Conditions~\ref{Meas}, \ref{Pconv}, \ref{Bound} and \ref{Diff}  are trivially satisfied. To justify the validity of the Condition~\ref{Conv} we take an arbitrary $\varepsilon>0$, and we choose $\delta=
\delta(\varepsilon)>0$ in such a way that the inequality \eqref{uniconv} holds both for $\vp(z,x,y)$ and for $\psi(z,x,y)$.
Then, using the inequality $$(a_1+a_2)(b_1+b_2)\leqslant 2(a_1b_1+a_2b_2)\quad\text{for}\quad  0<a_1<a_2,\quad 0<b_1<b_2,$$
we obtain
\begin{align*}
\vp\Big(\frac{t+s}2,x,y\Big)\psi\Big(\frac{t+s}2,x,y\Big)\leqslant &(1-\delta)^2\,\frac {\vp(s,x,y)+\vp(t,x,y)}2
\frac {\psi(s,x,y)+\psi(t,x,y)}2
\\
\leqslant&(1-\delta)^2\frac {\vp(s,x,y)\psi(s,x,y)+\vp(t,x,y)\psi(z,x,y)}2.
\end{align*}
Thus, the Condition \ref{Conv} holds for the product.

For each $\e>0$ we choose $\d=\d(\e)$ by the Condition~\ref{Conv}, and by this condition for the functions $\psi$ and $\phi$ we obtain
\begin{align*}
\vp\bigg(\psi\bigg(\frac{s+t}{2},x,y\bigg),x,y\bigg)\leqslant & \vp\bigg((1-\d)\frac{\psi(s,x,y)+\psi(t,x,y)}{2},x,y\bigg)
\\
\leqslant &
\frac{\vp\big((1-\d)\psi(s,x,y),x,y\big)+\vp\big((1-\d)\psi(t,x,y),x,y\big)}{2}.
\end{align*}
The convexity of the function $\vp$ yield
\begin{equation*}
\vp((1-\d)z,x,y)=\vp((1-\d)z+\d\cdot0,x,y)\leqslant (1-\d) \vp(z,x,y),
\end{equation*}
and this allows us to continue the above estimating
\begin{equation*}
\vp\bigg(\psi\bigg(\frac{s+t}{2},x,y\bigg),x,y\bigg)\leqslant   (1-\d)
\frac{\vp\big(\psi(s,x,y),x,y\big)+\vp\big(\psi(t,x,y),x,y\big)}{2}.
\end{equation*}
This proves the Condition~\ref{Conv} for the composition $\vp(\psi(z,x,y),x,y)$ and completes the proof.
\end{proof}

\begin{proof}[Proof of Lemma \ref{lem_sec_deriv}]
 Since $\vp(z,x,y)$ is twice continuously differentiable in the variable $z$, for all
positive $s$ and $t$ such that $0<s<t$ we have the representations
\begin{align*}
&
\vp(s,x,y)=\vp\Big(\frac{s+t}2,x,y\Big)-\vp'\Big(\frac{s+t}2,x,y\Big)\frac{t-s}2+
\frac18\vp''(\tilde{s},x,y)(t-s)^2, && s<\tilde{s}<\frac{s+t}2,
\\
&
\vp(t,x,y)=\vp\Big(\frac{s+t}2,x,y\Big)+\vp'\Big(\frac{s+t}2,x,y\Big)\frac{t-s}2+
\frac18\vp''(\tilde{t},x,y)(t-s)^2, &&  \frac{s+t}2<\tilde{t}<t,
\end{align*}
where $\tilde{s}$ and $\tilde{t}$ are some values.
Summing up these inequalities yields
\begin{align*}
\vp\Big(\frac{s+t}2,x,y\Big)=& \frac{\vp(s,x,y)+\vp(t,x,y)}2-\frac1{16}\vp''(\tilde{s},x,y)(t-s)^2-
\frac1{16}\vp''(\tilde{t},x,y)(t-s)^2
\\
\leqslant &
\frac{\vp(s,x,y)+\vp(t,x,y)}2-
\frac1{16}\vp''(\tilde{t},x,y)(t-s)^2
\end{align*}
Assuming that $t-s>\varepsilon t$ and considering inequality \eqref{secderbou}, we find
\begin{align*}
\vp\Big(\frac{s+t}2,x,y\Big)\leqslant & \frac{\vp(s,x,y)+\vp(t,x,y)}2-
\frac{1}{16}\vp''(\tilde{t},x,y)(t-s)^2
\\
\leqslant &
\frac{\vp(s,x,y)+\vp(t,x,y)}2-
\frac{c_7}{16}\vp(\tilde{t},x,y)\frac{(t-s)^2}{^(\tilde{t})^2}
\\
\leqslant &\frac{\vp(s,x,y)+\vp(t,x,y)}2-
\frac{c_7\varepsilon^2}{16}\vp(\tilde{t},x,y)
\end{align*}
According to \eqref{3.22a},
\begin{equation*}
\vp(\tilde{t},x,y)\geqslant \frac1{2^{p_+}\beta}\vp(t,x,y).
\end{equation*}
Therefore,
$$
\vp\Big(\frac{s+t}2,x,y\Big)\leqslant\frac{\vp(s,x,y)+\vp(t,x,y)}2-
\frac {c_7\varepsilon^2}{2^{4+p_+}\beta}\vp(t,x,y).
$$
This implies \eqref{uniconv} with $\delta= \frac{c_7\varepsilon^2}{\beta 2^{p_++5}\b}$. The proof is complete.
\end{proof}

\begin{proof}[Proof of Theorem~\ref{cor_pert}]
Since $\psi$
and $\vp$
satisfy Condition \ref{Meas}, the sum $\psi+\vp$
also does. The inequality (\ref{2.31}) for $\psi$ and the inequality (\ref{secderbou}) for $\vp$ imply immediately the inequality (\ref{secderbou}) for $\psi+\vp.$ By Lemma~\ref{lem_sec_deriv} this yields the Condition~\ref{Conv} for $\psi+\vp.$

It follows from  (\ref{2.32}), (\ref{2.31})
that
\begin{align*}
&|\psi'(z,x,y)|=\bigg|\psi'(0,x,y)+\int\limits_0^z\psi''(t,x,y)dt\bigg|\leqslant \int\limits_0^z|\psi''(t,x,y)|dt\leqslant c_8\vp'(z,x,y),
\\
&|\psi(z,x,y)|=\bigg|\psi(0,x,y)+\int\limits_0^z\psi'(t,x,y)dt\bigg|\leqslant \int\limits_0^z|\psi'(t,x,y)|dt\leqslant c_8\vp(z,x,y).
\end{align*}
Therefore,
\begin{equation*}
\vp(z,x,y)+\psi(z,x,y)\leqslant (1+c_8)\vp(z,x,y),\qquad
\vp'(z,x,y)+\psi'(z,x,y)\leqslant (1+c_8)\vp'(z,x,y)
\end{equation*}
which in turn imply that Conditions \ref{Pconv}--\ref{Diff} hold  for the sum $\vp+\psi.$
The proof is complete.
\end{proof}

\begin{proof}[Proof of Theorem~\ref{th2.11}]
The Conditions~\ref{Meas},~\ref{Pconv},~\ref{Bound},~\ref{Diff} are easily verified by direct calculations.  Using the estimates (\ref{secderbou}) and (\ref{2.33}), we obtain
\begin{align*}
(\psi(z,x,y)\vp(z,x,y))''=&\psi''(z,x,y)\vp(z,x,y)+2\psi'(z,x,y)\vp'(z,x,y)+ \psi(z,x,y)\vp''(z,x,y)
\\
\leqslant & 2\psi'(z,x,y)\vp'(z,x,y) -  c_9 z^{-1}\psi'(z,x,y)\vp(z,x,y)
\\
&+(c_7-c_{10}) z^{-2} \psi(z,x,y) \vp(z,x,y)
\\
=& 2z^{-1}\psi'(z,x,y)\vk(z,x,y) +(c_7-c_{10}) z^{-2} \psi(z,x,y) \vp(z,x,y),
\end{align*}
where
\begin{equation*}
\vk(z,x,y):=2z\vp'(z,x,y)-c_9 \vp(z,x,y).
\end{equation*}
The function $\vk(z,x,y)$ satisfies the identity $\vk(0,x,y)=0$ and the estimate
\begin{equation*}
\vk'(z,x,y)=2 z\vp''(z,x,y)+(2-c_9) \vp'(z,x,y)\geqslant 0,
\end{equation*}
and this is why $\vk$ is a non--negative function. Hence, we can continue the above estimating
\begin{equation*}
(\psi(z,x,y)\vp(z,x,y))'' \geqslant (c_7-c_{10}) z^{-2} \psi(z,x,y) \vp(z,x,y),
\end{equation*}
and by Lemma~\ref{lem_sec_deriv} the function $\psi(z,x,y)\vp(z,x,y)$ satisfies the Condition~\ref{Conv}. The proof is complete.
\end{proof}

\section*{Acknowledgments}

The authors thank R.N. Gumerov for useful discussions on general theory of dual spaces.

\end{document}